\documentclass[11pt,a4paper,reqno]{amsart}
\usepackage{color,latexsym,amsfonts,amssymb,bbm,comment}
\usepackage{hyperref}
\usepackage{amsmath,cite}
\usepackage{amsthm}
\usepackage{fullpage}
\usepackage{graphicx}
\usepackage{tikz}
\usepackage{caption,subcaption}
\usepackage{ifthen}
\usepackage{pgfplots}

\pgfmathsetseed{14}

\newcommand{\coxSegment}[6]{

	\pgfmathsetmacro{\Xa}{#1}
	\pgfmathsetmacro{\Xb}{#2}
	\pgfmathsetmacro{\Xc}{#3}
	\pgfmathsetmacro{\Xd}{#4}

	\draw (\Xa - \factor*\Xc + \factor*\Xa,\Xb  - \factor*\Xd + \factor*\Xb) -- (\Xc + \factor*\Xc - \factor*\Xa, \Xd + \factor*\Xd - \factor*\Xb);
	\foreach \y in {1,2,...,#5}{
		\pgfmathsetmacro{\Xf}{random()}
		\pgfmathsetmacro{\Xg}{1- \Xf)}
		\pgfmathsetmacro{\Xh}{\Xa*\Xf + \Xc*\Xg}
		\pgfmathsetmacro{\Xi}{\Xb*\Xf+ \Xd*\Xg}
		\pgfmathparse{\Xh*\Xh + \Xi*\Xi > #6 ? 1: 0}
		\ifthenelse{\pgfmathresult>0}{\fill (\Xh, \Xi) circle (2pt);}{}
	}

}

\newcommand{\coxSegmentBoxes}[3]{
	\foreach \y in {1,2,...,#1}{
		\pgfmathsetmacro{\Yf}{random()}
		\pgfmathsetmacro{\Yg}{random()}
		\pgfmathparse{\Yf*\Yf + \Yg*\Yg > #2/(#3*#3) ? 1: 0}
		\ifthenelse{\pgfmathresult>0}{\fill[gray] (#3*\Yf, #3*\Yg) circle (2pt);}{}	
	}
		\foreach \y in {1,2,...,#1}{
		\pgfmathsetmacro{\Yf}{random()}
		\pgfmathsetmacro{\Yg}{-random()}
		\pgfmathparse{\Yf*\Yf + \Yg*\Yg > #2/(#3*#3) ? 1: 0}
		\ifthenelse{\pgfmathresult>0}{\fill[gray] (#3*\Yf, #3*\Yg) circle (2pt);}{}	
	}
		\foreach \y in {1,2,...,#1}{
		\pgfmathsetmacro{\Yf}{-random()}
		\pgfmathsetmacro{\Yg}{random()}
		\pgfmathparse{\Yf*\Yf + \Yg*\Yg > #2/(#3*#3) ? 1: 0}
		\ifthenelse{\pgfmathresult>0}{\fill[gray] (#3*\Yf, #3*\Yg) circle (2pt);}{}	
	}
		\foreach \y in {1,2,...,#1}{
		\pgfmathsetmacro{\Yf}{-random()}
		\pgfmathsetmacro{\Yg}{-random()}
		\pgfmathparse{\Yf*\Yf + \Yg*\Yg > #2/(#3*#3) ? 1: 0}
		\ifthenelse{\pgfmathresult>0}{\fill[gray] (#3*\Yf, #3*\Yg) circle (2pt);}{}	
	}

}

\newcommand{\coxSegmentBoxesG}[3]{
	\foreach \y in {1,2,...,#1}{
		\pgfmathsetmacro{\Yf}{random()}
		\pgfmathsetmacro{\Yg}{random()}
		\pgfmathparse{\Yf*\Yf + \Yg*\Yg > #2/(#3*#3) ? 1: 0}
		\ifthenelse{\pgfmathresult>0}{\pgcircle{#3*\Yf}{#3*\Yg}{0.2cm};}{}	
	}
		\foreach \y in {1,2,...,#1}{
		\pgfmathsetmacro{\Yf}{random()}
		\pgfmathsetmacro{\Yg}{-random()}
		\pgfmathparse{\Yf*\Yf + \Yg*\Yg > #2/(#3*#3) ? 1: 0}
		\ifthenelse{\pgfmathresult>0}{\pgcircle{#3*\Yf}{#3*\Yg}{0.2cm};}{}	
	}
		\foreach \y in {1,2,...,#1}{
		\pgfmathsetmacro{\Yf}{-random()}
		\pgfmathsetmacro{\Yg}{random()}
		\pgfmathparse{\Yf*\Yf + \Yg*\Yg > #2/(#3*#3) ? 1: 0}
		\ifthenelse{\pgfmathresult>0}{\pgcircle{#3*\Yf}{#3*\Yg}{0.2cm};}{}	
	}
		\foreach \y in {1,2,...,#1}{
		\pgfmathsetmacro{\Yf}{-random()}
		\pgfmathsetmacro{\Yg}{-random()}
		\pgfmathparse{\Yf*\Yf + \Yg*\Yg > #2/(#3*#3) ? 1: 0}
		\ifthenelse{\pgfmathresult>0}{\pgcircle{#3*\Yf}{#3*\Yg}{0.2cm};}{}	
	}

}

\newcommand{\E}{\mathbb{E}} 
\newcommand{\Z}{\mathbb{Z}}
\newcommand{\R}{\mathbb{R}}

\def\eq{\begin{equation}}
\def\en{\end{equation}}

\newtheorem{theorem}{Theorem}[section]
\newtheorem{corollary}[theorem]{Corollary}

\newtheorem{lemma}[theorem]{Lemma}
\newtheorem{proposition}[theorem]{Proposition}
\theoremstyle{definition}
\newtheorem{definition}[theorem]{Definition}

\newtheorem{example}[theorem]{Example}

    \def\Q{{\Bbb Q}}

    \def\e{{\varepsilon}}

    \def\dist{{\rm dist}}

    \def\D{\Delta}
    \def\a{\alpha}
    \def\sm{\setminus}

    \def\e{\varepsilon}
    
    \def\phi{\varphi}
    \def\g{\gamma}
    \def\la{\lambda}
    \def\k{\kappa}
    
    \def\r{\rho}
    \def\de{\delta}

    \def\D{\Delta}
    \def\L{\Lambda}

    \def\P{{\Phi}}
    
    \def\T{\T}
    
    \def\X{\Xi}

    \def\C{{\mc C}}
    
    \def\V|{{\Vert}}

    \def\supp{\text{supp}}
    \def\essinf{\text{\rm ess-inf }}

    \def\d{{\rm d}}
    \def\E{\mathbb{E}}
    \def\V{\mathbb{V}}
    
    \def\es{\emptyset}
    
    \def\one{\mathbbmss{1}}
    \def\mc{\mathcal}
    
    \def\ms{\mathsf}
    
    \def\one{\mathbbmss{1}}
    \def\P{\mathbb{P}}
    \def\R{\mathbb{R}}
    \def\Z{\mathbb{Z}}

    \def\vXs{X_{\ms{S}}}
     \def\vXsi{X_{\ms{S}, i}}
    \def\vXl{X^{\lambda}}
    \def\vXls{X^{\lambda,*}}

    \def\vls{\lambda_{\ms{S}}}
    \def\vlc{\lambda_{c}}

    \def\vG{g_r}

    \def\Ce{\mc{C}^{\ms{ext}}_r}
    \def\Cee{\mc{C}^{\ms{ext}}_{r, \e}}

    \def\too{\rightarrow}
    \renewcommand{\to}{\uparrow}
    \def\tod{\downarrow}

\keywords{Cox processes, percolation, stabilization, large deviations}

\begin{document}
\author{Christian Hirsch}
\address[Christian Hirsch]{Mathematisches Institut, Ludwig-Maximilians-Universit\"at M\"unchen, 80333 Munich, Germany}
\email{hirsch@math.lmu.de} 
\author{Benedikt Jahnel}
\address[Benedikt Jahnel]{Weierstrass Institute for Applied Analysis and Stochastics, Mohrenstra\ss e 39, 10117 Berlin, Germany}
              \email{benedikt.jahnel@wias-berlin.de} 
\author{Elie Cali}
\address[Elie Cali]{Orange SA, 44 Avenue de la R\'epublique, 92326 Ch\^atillon, France}
              \email{elie.cali@orange.com} 

\title{Continuum percolation for Cox point processes}

\date{\today}

\begin{abstract}
We investigate continuum percolation for Cox point processes, that is, Poisson point processes driven by random intensity measures. First, we derive sufficient conditions for the existence of non-trivial sub- and super-critical percolation regimes based on the notion of stabilization. Second, we give asymptotic expressions for the percolation probability in large-radius, high-density and coupled regimes. 
In some regimes, we find universality, whereas in others, a sensitive dependence on the underlying random intensity measure survives. 
\end{abstract}

\maketitle

\section{Introduction}
Bernoulli bond percolation is one of the most prototypical models for the occurrence of phase transitions. Additionally, as of today, the continuum version of percolation where connections are formed according to distances in a spatial point process, has been investigated intensely in the Poisson case. More recently, the community has started to look at point processes that go far beyond the simplistic Poisson model. In particular, this includes sub-Poisson~\cite{bartekPerc1, bartekPerc2}, Ginibre~\cite{ginibrePerc} and Gibbsian point processes~\cite{gibbsPerc2,gibbsPerc1}.

Another stream of research that brought forward a variety of surprising results is the investigation of percolation processes living in a random environment. The seminal work on the critical probability for Voronoi percolation showed that dealing with random environments often requires the development of fundamental new methodological tools~\cite{vorPerc1,vorPerc2,vorPerc3}. Additionally, recent work on percolation in unimodular random graphs also revealed that fundamental properties of percolation on transitive graphs fail to carry over to the setting of random environments~\cite{unim}.

In light of these developments it comes as a surprise that continuum percolation for Cox point processes, i.e., Poisson point processes in a random environment,   have so far not been studied systematically. In this paper, we rectify this omission by providing conditions for the existence of a phase transition and by investigating the asymptotic behavior of the percolation probability in a number of different limiting regimes. 

\medskip
In addition to this mathematical motivation, our results have applications in the domain of telecommunication. Here, Cox processes are commonly employed for modelling various kinds of networks \cite[Chapter 5]{spodarev2013stochastic}. More precisely, for modelling the deployment of a telecommunication network, various random tessellation models for different types of street systems have been developed and tested against real data~\cite{gloaguen2006fitting}. The main idea of these models is to generate a random tessellation, with the same average characteristics as the street system, based on a planar Poisson point process. This could be a Voronoi, or Delaunay, or line tessellation, or it could be a more involved model like a nested tessellation \cite{delaunay}.

Once the street system is modelled, it is possible to add wireless users along the streets. The simplest way to do that is to use a linear Poisson point process along the streets. This will give rise to a Cox process. Building the Gilbert graph, i.e., drawing an edge between any two users with distance less than a given connection radius, one can obtain a very simplified model of users communicating via a Device-to-Device mechanism. Then, studying the percolation of this random graph, one could obtain results on the connectivity of the wireless network.

\medskip
The main results in this paper fall into two large categories: existence of phase transition and asymptotic analysis of percolation probabilities. First, we show that a variant of the celebrated concept of stabilization~\cite{yukLDP,normGeomProb, stab1, stab2, stab3} suffices to guarantee the existence of a sub-critical phase. In contrast, for the existence of a super-critical phase, stabilization alone is not enough since percolation is impossible unless the support of the random measure has sufficiently good connectivity properties itself. Hence, our proof for the existence of a super-critical phase relies on a variant of the notion of {asymptotic essential connectedness} from~\cite{aldous}.

Second, when considering the Poisson point process, the high-density or large-radius limit of the percolation probability tends to 1 exponentially fast and is governed by the isolation probability. In the random environment, the picture is more subtle since the regime of a large radius is no longer equivalent to that of a high density. Since we rely on a refined large-deviation analysis, we assume that the random environment is not only stabilizing, but in fact $b$-dependent. 

Since the high-density and the large-radius limit are no longer equivalent, this opens up the door to an analysis of coupled limits. As we shall see, the regime of a large radius and low density is of highly averaging nature and therefore results in a universal limiting behavior. On the other hand, in the converse limit the geometric structure of the random environment remains visible in the limit. In particular, a different scaling balance between the radius and density is needed when dealing with absolutely continuous and singular random measures, respectively. Finally, we illustrate our results with specific examples and simulations.

\section{Model definition and main results} 
Loosely speaking, Cox point processes are Poisson point processes in a random environment. More precisely, the random environment is given by a random element $\L$ in the space $\mathbb{M}$ of Borel measures on $\R^d$ equipped with the usual evaluation $\sigma$-algebra. 
Throughout the manuscript we assume that $\L$ is stationary, but at this point we do not impose any additional conditions. In particular, $\L$ could be an absolutely continuous or singular random intensity measure. Nevertheless, in some of the presented results, completely different behavior will appear. 
\begin{example}[Absolutely continuous environment]\label{Ex1}
    Let $\L(\d x) = \ell_x\d x$ with $\ell = \{\ell_x\}_{x \in \R^d}$ a stationary non-negative random field. For example, this includes random measures modulated by a random closed set $\Xi$, \cite[Section 5.2.2]{stoyanstochastic}. Here, $\ell_x=\la_1\one\{x\in \Xi\}+\la_2\one\{x\not\in \Xi\}$ with $\la_1,\la_2\ge0$. Another example are random measures induced by shot-noise fields, \cite[Section 5.6]{stoyanstochastic}. Here, $\ell_x=\sum_{X_i\in \vXs}k(x-X_i)$ for some non-negative integrable kernel $k:\, \R^d \to [0,\infty)$ with compact support and $\vXs$ a Poisson point process.
\end{example}

\begin{example}[Singular environment]\label{Ex2}
Let $\L = \nu_1(S\cap \d x)$ where $\nu_1$ denotes the one-dimensional Hausdorff measure and $S$ is a stationary segment process in $\R^d$. That is, $S$ is a stationary point process in the space of line segments~\cite[Chapter 8]{stoyanstochastic}. For example consider $S$ to be a Poisson-Voronoi, Poisson-Delaunay or a Poisson line  tessellation. 
\end{example}

Then, let $\vXl$ be a Cox process in $\R^d$ with stationary intensity measure $\la\L$ where $\la>0$ and $\E[\L([0,1]^d)] = 1$. That is, conditioned on $\L$, the point process $\vXl$ is a Poisson point process with intensity measure $\la\L$, see Figure~\ref{rlaFig}. 

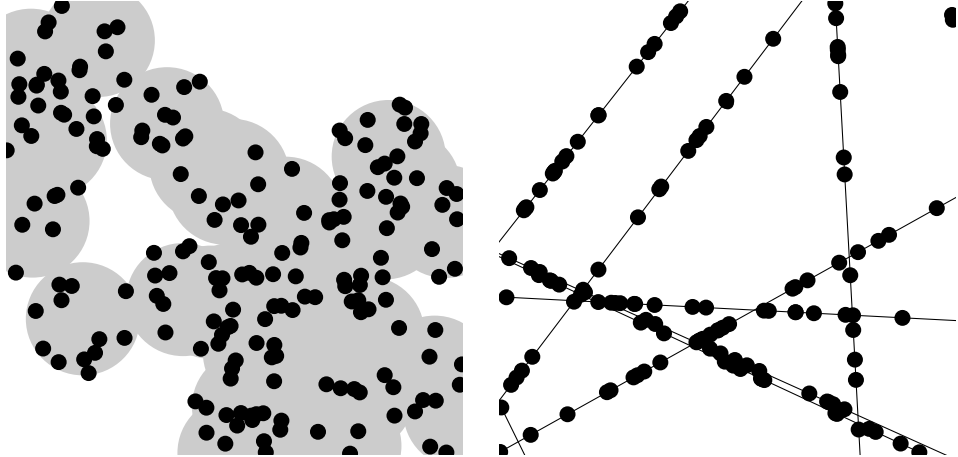
\begin{figure}[!htpb]
    \begin{tikzpicture}[scale = 1.5]
    \clip (-2,-2) rectangle (2,2);
    \fill[black!20!white]  (0.676,-0.461) circle (.5cm);
    \fill[black!20!white]  (-1.330,-0.804) circle (.5cm);
    \fill[black!20!white]  (-1.785,1.437) circle (.5cm);
    \fill[black!20!white]  (-0.450,-0.634) circle (.5cm);
    \fill[black!20!white]  (-0.082,0.351) circle (.5cm);
    \fill[black!20!white]  (-0.247,0.556) circle (.5cm);
    \fill[black!20!white]  (1.756,-1.276) circle (.5cm);
    \fill[black!20!white]  (1.524,0.216) circle (.5cm);
    \fill[black!20!white]  (1.353,0.630) circle (.5cm);
    \fill[black!20!white]  (-1.771,0.060) circle (.5cm);
    \fill[black!20!white]  (0.223,-1.103) circle (.5cm);
    \fill[black!20!white]  (1.484,0.368) circle (.5cm);
    \fill[black!20!white]  (1.322,0.044) circle (.5cm);
    \fill[black!20!white]  (-1.619,0.772) circle (.5cm);
    \fill[black!20!white]  (-0.591,0.920) circle (.5cm);
    \fill[black!20!white]  (1.075,-1.489) circle (.5cm);
    \fill[black!20!white]  (-1.198,1.656) circle (.5cm);
    \fill[black!20!white]  (-0.018,0.463) circle (.5cm);
    \fill[black!20!white]  (0.127,-1.581) circle (.5cm);
    \fill[black!20!white]  (1.166,-0.929) circle (.5cm);
    \fill[black!20!white]  (0.786,-0.071) circle (.5cm);
    \fill[black!20!white]  (-1.929,0.168) circle (.5cm);
    \fill[black!20!white]  (0.000,-1.989) circle (.5cm);
    \fill[black!20!white]  (-0.210,-0.657) circle (.5cm);
    \fill[black!20!white]  (1.851,0.056) circle (.5cm);
    \fill[black!20!white]  (0.965,-1.927) circle (.5cm);
    \fill[black!20!white]  (1.981,-1.625) circle (.5cm);
    \fill[black!20!white]  (0.664,-1.402) circle (.5cm);
    \fill[black!20!white]  (0.449,0.129) circle (.5cm);
    \fill[black!20!white]  (-0.003,-0.628) circle (.5cm);
    \fill  (-0.181,-0.827) circle (2pt);
    \fill  (1.584,-1.623) circle (2pt);
    \fill  (0.402,-1.785) circle (2pt);
    \fill  (1.496,1.062) circle (2pt);
    \fill  (-1.542,-1.187) circle (2pt);
    \fill  (-1.577,0.281) circle (2pt);
    \fill  (-0.966,1.309) circle (2pt);
    \fill  (1.655,-1.522) circle (2pt);
    \fill  (0.962,-0.459) circle (2pt);
    \fill  (1.633,0.835) circle (2pt);
    \fill  (-0.012,-0.723) circle (2pt);
    \fill  (-1.535,-0.502) circle (2pt);
    \fill  (-1.223,-1.105) circle (2pt);
    \fill  (-1.354,1.425) circle (2pt);
    \fill  (-1.523,1.204) circle (2pt);
    \fill  (0.339,-0.410) circle (2pt);
    \fill  (1.258,0.536) circle (2pt);
    \fill  (0.068,-0.412) circle (2pt);
    \fill  (-0.021,-1.243) circle (2pt);
    \fill  (-0.294,-1.069) circle (2pt);
    \fill  (-1.359,1.393) circle (2pt);
    \fill  (1.717,-1.934) circle (2pt);
    \fill  (-0.225,-0.303) circle (2pt);
    \fill  (0.147,-1.660) circle (2pt);
    \fill  (-1.517,-0.640) circle (2pt);
    \fill  (1.166,0.326) circle (2pt);
    \fill  (-1.185,-0.984) circle (2pt);
    \fill  (0.538,-0.429) circle (2pt);
    \fill  (1.394,-1.409) circle (2pt);
    \fill  (0.922,0.249) circle (2pt);
    \fill  (1.337,-0.003) circle (2pt);
    \fill  (-0.034,-0.867) circle (2pt);
    \fill  (-0.539,0.975) circle (2pt);
    \fill  (1.318,-1.302) circle (2pt);
    \fill  (0.058,0.024) circle (2pt);
    \fill  (0.210,0.028) circle (2pt);
    \fill  (-1.494,1.000) circle (2pt);
    \fill  (-0.451,-0.208) circle (2pt);
    \fill  (-1.026,1.773) circle (2pt);
    \fill  (-0.570,-0.400) circle (2pt);
    \fill  (-1.514,1.961) circle (2pt);
    \fill  (-1.887,1.270) circle (2pt);
    \fill  (0.708,-0.613) circle (2pt);
    \fill  (1.733,-0.188) circle (2pt);
    \fill  (0.872,0.076) circle (2pt);
    \fill  (0.208,0.385) circle (2pt);
    \fill  (0.510,-0.727) circle (2pt);
    \fill  (-1.244,1.164) circle (2pt);
    \fill  (-1.736,1.257) circle (2pt);
    \fill  (1.762,-0.903) circle (2pt);
    \fill  (0.145,-0.079) circle (2pt);
    \fill  (0.412,-1.704) circle (2pt);
    \fill  (-0.059,-0.879) circle (2pt);
    \fill  (-0.109,-0.946) circle (2pt);
    \fill  (1.457,0.215) circle (2pt);
    \fill  (0.411,-0.689) circle (2pt);
    \fill  (-1.865,0.906) circle (2pt);
    \fill  (-0.101,0.207) circle (2pt);
    \fill  (1.285,-0.265) circle (2pt);
    \fill  (1.859,-1.986) circle (2pt);
    \fill  (1.051,-1.425) circle (2pt);
    \fill  (-1.207,0.721) circle (2pt);
    \fill  (-1.678,-1.067) circle (2pt);
    \fill  (0.131,-0.398) circle (2pt);
    \fill  (-1.318,-1.163) circle (2pt);
    \fill  (-1.861,0.027) circle (2pt);
    \fill  (1.429,0.632) circle (2pt);
    \fill  (1.402,0.443) circle (2pt);
    \fill  (0.057,-1.637) circle (2pt);
    \fill  (-1.371,0.355) circle (2pt);
    \fill  (-1.743,-0.735) circle (2pt);
    \fill  (0.926,0.395) circle (2pt);
    \fill  (-0.624,-0.673) circle (2pt);
    \fill  (-1.139,1.737) circle (2pt);
    \fill  (-0.174,0.070) circle (2pt);
    \fill  (0.328,-1.145) circle (2pt);
    \fill  (0.946,-0.110) circle (2pt);
    \fill  (1.977,-1.387) circle (2pt);
    \fill  (1.111,-0.423) circle (2pt);
    \fill  (1.148,0.731) circle (2pt);
    \fill  (1.490,0.918) circle (2pt);
    \fill  (1.029,-0.651) circle (2pt);
    \fill  (-0.952,-0.560) circle (2pt);
    \fill  (-0.451,0.789) circle (2pt);
    \fill  (0.579,-0.174) circle (2pt);
    \fill  (0.157,-1.699) circle (2pt);
    \fill  (1.330,0.274) circle (2pt);
    \fill  (1.176,-0.720) circle (2pt);
    \fill  (1.432,0.135) circle (2pt);
    \fill  (1.110,-0.746) circle (2pt);
    \fill  (1.014,-1.996) circle (2pt);
    \fill  (0.036,0.242) circle (2pt);
    \fill  (0.831,0.068) circle (2pt);
    \fill  (-1.782,0.812) circle (2pt);
    \fill  (0.253,-1.639) circle (2pt);
    \fill  (0.972,-0.516) circle (2pt);
    \fill  (-1.721,1.081) circle (2pt);
    \fill  (0.348,-1.356) circle (2pt);
    \fill  (-0.033,-1.333) circle (2pt);
    \fill  (-1.427,-0.514) circle (2pt);
    \fill  (-0.304,1.292) circle (2pt);
    \fill  (0.024,-1.748) circle (2pt);
    \fill  (0.192,-0.441) circle (2pt);
    \fill  (1.953,0.076) circle (2pt);
    \fill  (-0.105,-0.446) circle (2pt);
    \fill  (-0.395,-0.162) circle (2pt);
    \fill  (1.601,0.439) circle (2pt);
    \fill  (-0.312,0.281) circle (2pt);
    \fill  (-1.553,0.292) circle (2pt);
    \fill  (-1.155,0.698) circle (2pt);
    \fill  (1.862,0.352) circle (2pt);
    \fill  (0.269,-0.807) circle (2pt);
    \fill  (-0.727,1.176) circle (2pt);
    \fill  (-0.070,-1.655) circle (2pt);
    \fill  (-0.080,-1.908) circle (2pt);
    \fill  (0.807,-1.383) circle (2pt);
    \fill  (-0.245,-1.811) circle (2pt);
    \fill  (-0.471,0.475) circle (2pt);
    \fill  (0.192,-1.647) circle (2pt);
    \fill  (-1.205,0.785) circle (2pt);
    \fill  (-1.895,1.158) circle (2pt);
    \fill  (1.450,1.089) circle (2pt);
    \fill  (0.419,-0.222) circle (2pt);
    \fill  (0.833,0.047) circle (2pt);
    \fill  (0.194,-1.018) circle (2pt);
    \fill  (-0.609,0.999) circle (2pt);
    \fill  (0.958,0.097) circle (2pt);
    \fill  (0.733,-1.803) circle (2pt);
    \fill  (1.471,0.186) circle (2pt);
    \fill  (1.086,-0.640) circle (2pt);
    \fill  (1.099,-1.455) circle (2pt);
    \fill  (1.795,-0.433) circle (2pt);
    \fill  (0.971,0.792) circle (2pt);
    \fill  (-1.040,1.086) circle (2pt);
    \fill  (1.711,-1.139) circle (2pt);
    \fill  (-1.386,0.874) circle (2pt);
    \fill  (1.934,-0.362) circle (2pt);
    \fill  (-1.753,0.214) circle (2pt);
    \fill  (1.639,0.918) circle (2pt);
    \fill  (-1.998,0.685) circle (2pt);
    \fill  (-1.734,1.266) circle (2pt);
    \fill  (0.366,-1.132) circle (2pt);
    \fill  (0.349,-0.692) circle (2pt);
    \fill  (-0.632,0.728) circle (2pt);
    \fill  (-0.818,0.801) circle (2pt);
    \fill  (0.932,-1.418) circle (2pt);
    \fill  (-0.605,-0.925) circle (2pt);
    \fill  (0.618,-0.608) circle (2pt);
    \fill  (0.923,0.860) circle (2pt);
    \fill  (-1.279,-1.284) circle (2pt);
    \fill  (1.949,0.299) circle (2pt);
    \fill  (-1.234,0.985) circle (2pt);
    \fill  (-0.431,0.809) circle (2pt);
    \fill  (0.611,0.132) circle (2pt);
    \fill  (-0.141,-1.025) circle (2pt);
    \fill  (-1.901,1.496) circle (2pt);
    \fill  (-0.696,-0.422) circle (2pt);
    \fill  (-1.128,1.560) circle (2pt);
    \fill  (1.996,-1.207) circle (2pt);
    \fill  (1.404,-1.660) circle (2pt);
    \fill  (-1.661,1.736) circle (2pt);
    \fill  (0.350,-1.035) circle (2pt);
    \fill  (0.259,-1.886) circle (2pt);
    \fill  (-0.681,-0.601) circle (2pt);
    \fill  (1.100,-0.505) circle (2pt);
    \fill  (-0.441,1.244) circle (2pt);
    \fill  (-0.964,-0.973) circle (2pt);
    \fill  (-0.808,0.860) circle (2pt);
    \fill  (-0.161,-0.444) circle (2pt);
    \fill  (1.318,0.568) circle (2pt);
    \fill  (0.185,0.667) circle (2pt);
    \fill  (0.275,-1.985) circle (2pt);
    \fill  (-1.592,-0.013) circle (2pt);
    \fill  (-1.668,1.361) circle (2pt);
    \fill  (-0.654,0.743) circle (2pt);
    \fill  (-0.342,-1.534) circle (2pt);
    \fill  (1.584,0.763) circle (2pt);
    \fill  (-1.917,-0.396) circle (2pt);
    \fill  (1.086,-1.798) circle (2pt);
    \fill  (1.329,-0.636) circle (2pt);
    \fill  (1.766,-1.528) circle (2pt);
    \fill  (-0.131,-0.553) circle (2pt);
    \fill  (1.300,-0.437) circle (2pt);
    \fill  (0.011,-1.172) circle (2pt);
    \fill  (-0.246,-1.590) circle (2pt);
    \fill  (-0.706,-0.221) circle (2pt);
    \fill  (0.587,-0.133) circle (2pt);
    \fill  (1.824,0.203) circle (2pt);
    \fill  (0.505,0.521) circle (2pt);
    \fill  (-1.630,1.816) circle (2pt);
    \fill  (-1.544,1.302) circle (2pt);
    \fill  (1.169,0.954) circle (2pt);
    \fill  (1.444,-0.885) circle (2pt);
    \fill  (-1.521,1.017) circle (2pt);

       \end{tikzpicture}
    \;~
    \begin{tikzpicture}[scale = 1.5]
    \clip (-2,-2) rectangle (2,2);
    \fill (0.000,1.100) circle (1.50pt);
    \draw (3.601,-3.469)--(4.353,2.461);
    \draw (4.021,-2.972)--(3.551,-3.520);
    \draw (3.937,3.082)--(1.963,4.598);
    \draw (-3.966,-3.044)--(-3.048,-3.964);
    \draw (0.782,4.938)--(1.314,-4.824);
    \draw (3.959,-3.053)--(-4.879,1.093);
    \draw (-4.118,-2.835)--(1.742,4.687);
    \draw (-4.793,1.424)--(-3.688,3.377);
    \draw (-4.882,1.078)--(4.024,-2.967);
    \draw (-4.135,2.811)--(-0.326,-4.989);
    \draw (-2.544,4.305)--(-4.416,-2.346);
    \draw (-4.797,1.410)--(3.122,3.906);
    \draw (4.325,2.508)--(3.073,3.944);
    \draw (-4.979,-0.454)--(4.905,-0.972);
    \draw (-2.332,4.423)--(1.736,4.689);
    \draw (2.428,4.371)--(-3.524,-3.547);
    \draw (-3.935,3.084)--(-3.724,-3.336);
    \draw (1.690,4.706)--(3.815,-3.233);
    \draw (4.979,-0.460)--(2.041,4.565);
    \draw (4.675,1.772)--(-3.943,-3.075);
    \draw (-4.765,-1.514)--(-2.291,4.444);
    \draw (2.500,4.770)--(2.926,-4.055);
    \draw (-0.634,-4.960)--(3.001,-3.999);
    \draw (4.681,-1.759)--(-3.955,-3.059);
    \draw (0.003,-5.000)--(1.685,-4.707);
    \fill (0.965,1.575) circle (2pt);
    \fill (0.968,1.530) circle (2pt);
    \fill (1.148,-1.783) circle (2pt);
    \fill (0.968,1.517) circle (2pt);
    \fill (1.074,-0.419) circle (2pt);
    \fill (1.100,-0.903) circle (2pt);
    \fill (0.943,1.983) circle (2pt);
    \fill (1.115,-1.165) circle (2pt);
    \fill (1.124,-1.344) circle (2pt);
    \fill (1.025,0.472) circle (2pt);
    \fill (0.986,1.201) circle (2pt);
    \fill (0.964,1.599) circle (2pt);
    \fill (0.950,1.853) circle (2pt);
    \fill (1.017,0.622) circle (2pt);
    \fill (0.111,-1.248) circle (2pt);
    \fill (-1.550,-0.469) circle (2pt);
    \fill (1.516,-1.907) circle (2pt);
    \fill (1.680,-1.984) circle (2pt);
    \fill (0.320,-1.346) circle (2pt);
    \fill (-1.479,-0.502) circle (2pt);
    \fill (-0.560,-0.933) circle (2pt);
    \fill (1.297,-1.804) circle (2pt);
    \fill (0.958,-1.645) circle (2pt);
    \fill (0.299,-1.336) circle (2pt);
    \fill (0.944,-1.639) circle (2pt);
    \fill (-0.026,-1.183) circle (2pt);
    \fill (0.048,-1.218) circle (2pt);
    \fill (-1.657,-0.418) circle (2pt);
    \fill (1.239,-1.777) circle (2pt);
    \fill (0.291,-1.332) circle (2pt);
    \fill (-1.561,-0.464) circle (2pt);
    \fill (-0.763,-0.838) circle (2pt);
    \fill (-1.288,-0.591) circle (2pt);
    \fill (-1.774,0.174) circle (2pt);
    \fill (-1.517,0.504) circle (2pt);
    \fill (-0.699,1.553) circle (2pt);
    \fill (-1.416,0.634) circle (2pt);
    \fill (-1.316,0.761) circle (2pt);
    \fill (-1.535,0.480) circle (2pt);
    \fill (-1.136,0.993) circle (2pt);
    \fill (-0.499,1.810) circle (2pt);
    \fill (-0.418,1.914) circle (2pt);
    \fill (-0.799,1.425) circle (2pt);
    \fill (-1.789,0.155) circle (2pt);
    \fill (-1.133,0.997) circle (2pt);
    \fill (-1.525,0.493) circle (2pt);
    \fill (-1.649,0.334) circle (2pt);
    \fill (-1.452,0.586) circle (2pt);
    \fill (-1.767,0.182) circle (2pt);
    \fill (-0.642,1.626) circle (2pt);
    \fill (-0.453,1.869) circle (2pt);
    \fill (-1.650,-0.390) circle (2pt);
    \fill (-0.064,-1.110) circle (2pt);
    \fill (0.924,-1.559) circle (2pt);
    \fill (-1.253,-0.570) circle (2pt);
    \fill (-1.918,-0.268) circle (2pt);
    \fill (-1.726,-0.355) circle (2pt);
    \fill (1.023,-1.604) circle (2pt);
    \fill (0.871,-1.535) circle (2pt);
    \fill (0.715,-1.464) circle (2pt);
    \fill (0.167,-1.215) circle (2pt);
    \fill (-0.637,-0.850) circle (2pt);
    \fill (0.062,-1.168) circle (2pt);
    \fill (-0.718,-0.813) circle (2pt);
    \fill (0.886,-1.542) circle (2pt);
    \fill (-0.156,-1.068) circle (2pt);
    \fill (0.276,-1.265) circle (2pt);
    \fill (-1.987,-1.588) circle (2pt);
    \fill (-1.942,-0.614) circle (2pt);
    \fill (-0.191,-0.705) circle (2pt);
    \fill (-1.136,-0.656) circle (2pt);
    \fill (0.317,-0.732) circle (2pt);
    \fill (-1.023,-0.662) circle (2pt);
    \fill (1.532,-0.796) circle (2pt);
    \fill (-0.821,-0.672) circle (2pt);
    \fill (0.755,-0.755) circle (2pt);
    \fill (-0.987,-0.664) circle (2pt);
    \fill (1.098,-0.773) circle (2pt);
    \fill (0.595,-0.746) circle (2pt);
    \fill (0.594,-0.746) circle (2pt);
    \fill (-0.309,-0.699) circle (2pt);
    \fill (-0.944,-0.666) circle (2pt);
    \fill (1.033,-0.769) circle (2pt);
    \fill (0.359,-0.734) circle (2pt);
    \fill (-0.642,-0.682) circle (2pt);
    \fill (-0.813,-0.673) circle (2pt);
    \fill (-1.135,-0.369) circle (2pt);
    \fill (0.398,1.671) circle (2pt);
    \fill (-0.014,1.122) circle (2pt);
    \fill (-1.853,-1.323) circle (2pt);
    \fill (-1.715,-1.140) circle (2pt);
    \fill (-1.901,-1.388) circle (2pt);
    \fill (-0.275,0.775) circle (2pt);
    \fill (-0.247,0.812) circle (2pt);
    \fill (-0.788,0.093) circle (2pt);
    \fill (-0.584,0.365) circle (2pt);
    \fill (-1.806,-1.262) circle (2pt);
    \fill (-1.269,-0.547) circle (2pt);
    \fill (-0.601,0.341) circle (2pt);
    \fill (0.146,1.335) circle (2pt);
    \fill (-0.188,0.892) circle (2pt);
    \fill (-1.349,-0.653) circle (2pt);
    \fill (-0.347,0.680) circle (2pt);
    \fill (0.592,-0.524) circle (2pt);
    \fill (-0.241,-0.993) circle (2pt);
    \fill (-0.592,-1.190) circle (2pt);
    \fill (-0.275,-1.012) circle (2pt);
    \fill (-0.220,-0.981) circle (2pt);
    \fill (1.142,-0.215) circle (2pt);
    \fill (-0.154,-0.944) circle (2pt);
    \fill (-0.085,-0.905) circle (2pt);
    \fill (0.977,-0.308) circle (2pt);
    \fill (1.414,-0.062) circle (2pt);
    \fill (0.568,-0.538) circle (2pt);
    \fill (-0.793,-1.303) circle (2pt);
    \fill (1.833,0.174) circle (2pt);
    \fill (0.011,-0.851) circle (2pt);
    \fill (-1.997,-1.981) circle (2pt);
    \fill (-1.029,-1.436) circle (2pt);
    \fill (-0.045,-0.883) circle (2pt);
    \fill (-0.733,-1.270) circle (2pt);
    \fill (-0.827,-1.322) circle (2pt);
    \fill (0.699,-0.464) circle (2pt);
    \fill (1.321,-0.115) circle (2pt);
    \fill (-0.253,-0.999) circle (2pt);
    \fill (-0.088,-0.907) circle (2pt);
    \fill (-1.060,-1.453) circle (2pt);
    \fill (-1.406,-1.648) circle (2pt);
    \fill (-1.729,-1.830) circle (2pt);
    \fill (1.967,1.882) circle (2pt);
    \fill (1.974,1.839) circle (2pt);

   \end{tikzpicture}
    \caption{
    Realizations of Cox point processes based on absolutely continuous (left) and singular (right) random measures.}
    
    \label{rlaFig}
\end{figure}

To study continuum percolation on $\vXl$, we work with the Gilbert graph $\vG(\vXl)$ on the vertex set $\vXl$ where two points $X_i, X_j \in \vXl$ are connected by an edge if their distance is less than a connection threshold $r>0$. The graph $\vG(\vXl)$ percolates if it contains an infinite connected component.

\subsection{Phase transitions}
First, we establish sufficient criteria for a non-trivial phase transition of continuum percolation in Cox processes. More precisely, we let 
$$\vlc = \vlc(r) = \inf\{\la:\, \P(\vG(\vXl)\text{ percolates})>0\},$$
denote the \emph{critical intensity} for continuum percolation. In contrast to the Poisson case, in the Cox setting the non-triviality of the phase transition, i.e., $0 < \vlc <\infty$, may fail without any further assumptions on $\L$~\cite[Example 4.1]{bartekPerc1}.  For our results we therefore assume that $\L$ exhibits weak spatial correlations in the spirit of stabilization~\cite{PeWa08}. To make this precise, we write $\L_B$ to indicate the restriction of the random measure $\L$ to the set $B \subset \R^d$. Further write 
$$Q_n(x) = x + [-n/2, n/2]^d$$ 
for the cube with side length $n\ge1$ centered at $x \in \R^d$ and put $Q_n = Q_n(o)$. We write $\dist(\varphi, \psi)=\inf\{|x-y|:\, x\in\varphi, y\in\psi\}$ to denote the distance between sets $\varphi,\psi\subset\R^d$.

\begin{definition}\label{Stab}
The random measure $\L$ is \emph{stabilizing}, if there exists a random field of \textit{stabilization radii} $R=\{R_x\}_{x\in\R^d}$ defined on the same probability space as $\L$ such that
\begin{enumerate}
    \item $(\L, R)$ are jointly stationary,
    \item $\lim_{n\uparrow\infty}\P(\sup_{y\in Q_n \cap\,  \Q^d}R_y < n)=1$, and
\item for all $n \ge 1$, the random variables
    $$\Big\{f(\L_{Q_{n}(x)})\one\{\sup_{y \in Q_n(x) \cap \, \Q^d}R_y < n\}\Big\}_{x \in \varphi}$$
        are independent for all  bounded measurable functions $f:\,\mathbb{M} \too [0,\infty)$ and finite $\varphi \subset \R^d$ with $\dist(x, \varphi \setminus \{x\}) > 3n$ for all $x \in \varphi$.
\end{enumerate}
\end{definition}
A strong form of stabilization is given if $\L$ is \emph{$b$-dependent} in the sense that $\L_A$ and $\L_B$ are independent whenever $\dist(A,B)>b$. The above Example~\ref{Ex1} is $b$-dependent for example in cases where $\Xi$ is the classical Poisson-Boolean model and for the shot-noise field. Further, the Poisson-Voronoi tessellation, as in Example~\ref{Ex2}, is stabilizing, as we see in Section~\ref{Ex}. 

In order to avoid confusion, note that the literature contains several different forms of stabilization. Our definition is in the spirit of internal stabilization~\cite{PeWa08}. Loosely speaking, the configuration of the measure in a neighborhood of $x$ does not depend on the configuration of the measure in the neighborhood of points $y$ with $|x-y| > R_x$. The notion of external stabilization would additionally include that conversely the configuration of the measure around $x$ does not affect the measure around $y$.

Stabilization implies the existence of sub-critical phase.
\begin{theorem}
\label{subCritStabThm}
    If $r > 0$ and $\L$ is stabilizing, then $\vlc(r) > 0$.
\end{theorem}

For the existence of a super-critical regime, the condition that $\L$ is stabilizing is not sufficient. For example, the measure $\L\equiv0$ is stabilizing, but $\vlc(r)=\infty$ for every $r > 0$. Consequently, we rely on the idea of asymptotic essential connectedness (see~\cite{aldous}) to introduce a sufficient condition for the existence of a super-critical phase. To state this succinctly, we write 
$$\supp(\mu) = \{x \in \R^d:\, \mu(Q_\e(x)) > 0 \text{ for every $\e > 0$} \}$$
for the support of a measure $\mu$.

\begin{definition}
    \label{aecDef}
		A stabilizing random measure $\L$ with stabilization radii $R$ is \emph{asymptotically essentially connected} if for all $n \ge 1$, whenever $\sup_{y\in Q_{2n}}R_y<n/2$ we have that 
        \begin{enumerate}
            \item $\supp(\L_{Q_n}) \ne \es$ and 
            \item $\supp(\L_{Q_n})$ is contained in a connected component of $\supp(\L_{Q_{2n}})$.
        \end{enumerate}
	\end{definition}
Example~\ref{Ex1} for $\Xi$ the Poisson-Boolean model and with $\la_1,\la_2>0$ as well as Example~\ref{Ex2} for the Poisson-Voronoi and Poisson-Delaunay tessellation are asymptotically essentially connected. For the Poisson-Boolean model this is clear and for the Poisson-Voronoi tessellation case we provide a detailed proof in Section~\ref{Ex}. The shot-noise field is not asymptotically essentially connected in general. Under the assumption of asymptotic essential connectedness, there is a non-trivial super-critical phase.

\begin{theorem}
\label{supCritThm}
    If $r > 0$ and $\L$ is asymptotically essentially connected, then $\vlc(r) < \infty$.
\end{theorem}

\subsection{Asymptotic results on the percolation probability}
In classical continuum percolation based on a homogeneous Poisson point process, the critical intensity is characterized via the percolation probability in the sense that 
$$\vlc=\inf\{\la:\, \P(o\leftrightsquigarrow \infty)>0\},$$
where $\{ o \leftrightsquigarrow \infty\}$ denotes the event that $o$ is contained in an infinite connected component of the Gilbert graph $\vG(\vXl)\cup\{o\}$. The reason for this identity is the equivalence of the Poisson point process and its reduced Palm version~\cite{poisBook}. This is no longer true for general Cox processes and therefore, the proper definition of the percolation probability relies on Palm calculus.  The \emph{Palm version} $\vXls$ of $\vXl$ is a point process whose distribution is defined via 
\begin{align*}
\E[f(\vXls)] = \frac{1}{\la}\E\Big[\sum_{X_i\in Q_1}f(\vXl-X_i)\Big]
\end{align*}
where $f: \mathbb{M}_{\ms{co}}\too [0, \infty)$ is any bounded measurable function acting on the set of $\sigma$-finite counting measures $\mathbb{M}_{\ms{co}}$. In particular,  $\P(o \in \vXls)=1$. Now, 
$$\theta(\la,r) = \P(o \leftrightsquigarrow \infty\text{ in }\vG(\vXls))$$
denotes the \emph{percolation probability} (of the origin). With this definition we recover the identity $\vlc = \inf_{\la>0}\theta(\la,r)$. Indeed, if $\theta(\la, r) > 0$, then 
$\E[\#\{X_i\in Q_1:\, X_i\leftrightsquigarrow \infty\}] > 0$
and hence 
$\P(\vG(\vXl)\text{ percolates}) > 0$. Conversely, if $\theta(\la,r)=0$, then 
$\E[\#\{X_i\in Q_1:\, X_i\leftrightsquigarrow \infty\}] = 0$
and hence 
$\P(\vG(\vXl)\text{ percolates}) = 0$ by stationarity.

Note that the Palm version $\vXls$ is the union of the origin and another Cox process defined via the Palm version $\L^*$ of the original random measure, see~\cite{fleischer}. Finally, the translation operator $\vartheta$ is defined by $\vartheta_x(\L(A))=\L(A+x)$ for all measurable $A\subset\R^d$. The distribution of $\L^*$ is given by
\begin{align*}
    \E[f(\L^*)] = \E\Big[\int_{Q_1}\L(\d x)f(\vartheta_x(\L))\Big]
\end{align*}
where $f: \mathbb{M}\too [0, \infty)$ is any bounded measurable function. 

%%%%%%%%%%%%%%%%%%%%%%%%%%%%%%%%%%%%%%%%%%%
%%%%%%%%%%%%%%%%%%%%%%%%%%%%%%%%%%%%%%%%%%%
%%%%%%HIGH-DENSITY LIMITS
%%%%%%%%%%%%%%%%%%%%%%%%%%%%%%%%%%%%%%%%%%%
%%%%%%%%%%%%%%%%%%%%%%%%%%%%%%%%%%%%%%%%%%%
\subsubsection{Large-radius and high-density limits}
In the Poisson-Boolean model, the percolation probabilities approach 1 exponentially fast as the radius grows large~\cite{pcPerc}.  More precisely, 
\begin{align}\label{PPPAs}
\lim_{r\uparrow\infty}r^{-d}\log(1 - \theta(\la,r)) = -|B_1(o)|\la,
\end{align}
where $|B_s(x)|$ denotes the volume of the $d$-dimensional ball with radius $s > 0$ centered at $x \in \R^d$. For $b$-dependent Cox processes, the exponentially fast convergence remains valid with a $\L$-dependent rate.

\begin{theorem}\label{rlowLDPThm}
    If $\la > 0$, then
    $$\liminf_{r \uparrow \infty} r^{-d} \log(1 - \theta(\la,r)) \ge \liminf_{r \to \infty} r^{-d} \log\E[\exp(-\la \L^*(B_r(o)))].$$
    If, additionally, $\L$ is $b$-dependent and $\L(Q_1)$ has all exponential moments, then
    the limit 
    $$I^* =- \lim_{r \to \infty} r^{-d} \log\E[\exp(-\la \L(Q_r))]$$
    exists and 
    $$\lim_{r \uparrow \infty} r^{-d} \log(1 - \theta(\la,r)) = -|B_1(o)|I^*.$$
\end{theorem}
In terms of large deviations, Theorem~\ref{rlowLDPThm} states that the most efficient way to avoid percolation is to force the origin to be isolated, see the left side of Figure~\ref{rlowFig}. 
\begin{figure}[!htpb]
    \centering
    \includegraphics{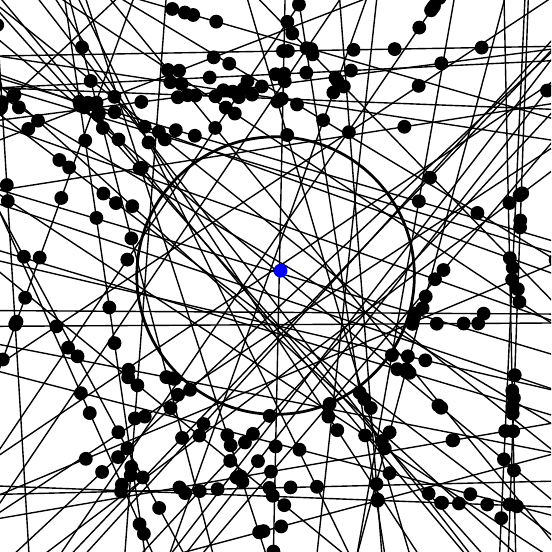}
    \;~
    \includegraphics{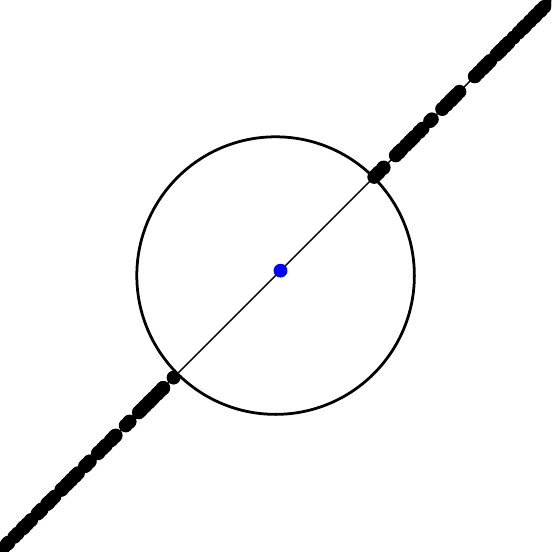}
    \caption{Asymptotic behavior in large-radius (left) and high-density limit (right) for the example of Poisson line tessellation.}
    \label{rlowFig}
\end{figure}
The above result is applicable to Example~\ref{Ex1} in case of the Poisson-Boolean model and the shot-noise field. To the best of our knowledge, for the classical tessellation processes, as in Example~\ref{Ex2}, the large-deviation behavior for the total edge length has not been derived in literature yet. Therefore, in this situation, the above result only gives a lower bound. Although computing the limiting Laplace transform is difficult in general, we show in Section~\ref{Ex} that for the shot-noise field, the original expression simplifies substantially.
\medskip

In classical continuum percolation, the scaling invariance of the Poisson-Boolean model makes it possible to translate limiting statements for $r \uparrow \infty$ into statements for $\la\uparrow\infty$. This is no longer the case for Cox processes. In the high-density regime, the connectivity structure of the support of $\L^*$ becomes apparent. Loosely speaking, here the rate function is given by the most efficient way to avoid percolation. For a given realization of $\L^*$, percolation can be avoided by finite clusters at the origin such that there are no points at distance $r$ from the cluster. In a second step, we optimize over $\L^*$, see the right side of Figure~\ref{rlowFig}.

More precisely, let the family $\mc{R}_r$ consist of all compact sets that contain the origin and are \emph{$r$-connected}. That is, of all compact $A \subset \R^d$ such that $o \in A$ and such that the points at distance at most $r/2$ from $A$ form a connected subset of $\R^d$. Moreover, let
$$\partial_r A = \{x \in \R^d:\, \dist(x, A) < r\} \setminus A$$
denote the \emph{$r$-boundary} of $A$. The next result characterizes the asymptotic behavior of the percolation probability for stabilizing random measures that are supported on $\R^d$.
\begin{theorem}\label{lalowLDPThm}
    Let $r>0$. Then,
    $$\liminf_{\la\uparrow\infty}\la^{-1}\log(1 - \theta(\la, r))\ge- \inf_{A \in \mc{R}_r}\essinf \L^*(\partial_r A).$$
    If, additionally, $\L$ is $b$-dependent and $\essinf \L(Q_\delta) > 0$ for every $\delta > 0$, then
    $$\limsup_{\la\uparrow\infty}\la^{-1}\log(1 - \theta(\la, r))\le -\lim_{\e\tod0}\inf_{A \in \mc{R}_{r + \e}}\essinf \L^*(\partial_{r - \e} A).$$
\end{theorem}
In general it is not true that the lower bound given by the isolation probability describes the true rate of decay of $1 - \theta(\la,r)$. Indeed, if the support of $\L$ does not percolate, then $\sup_{\la > 0}\theta(\la,r) = 0$. Nevertheless, for Example~\ref{Ex1} in case of the Poisson-Boolean model with $\la_1,\la_2>0$ the above right-hand sides are optimal for $A = \{o\}$. For the singular examples and for the shot-noise example the condition $\essinf \L(Q_\delta) > 0$ for every $\delta > 0$ is not satisfied. The right-hand side of the lower bound can be computed for the Poisson-Voronoi tessellation case and equals $-2r$.

\medskip
%%%%%%%%%%%%%%%%%%%%%%%%%%%%%%%%%%%%%%%%%%%
%%%%%%%%%%%%%%%%%%%%%%%%%%%%%%%%%%%%%%%%%%%
%%%%%%COUPLED LIMITS
%%%%%%%%%%%%%%%%%%%%%%%%%%%%%%%%%%%%%%%%%%%
%%%%%%%%%%%%%%%%%%%%%%%%%%%%%%%%%%%%%%%%%%%
\subsubsection{Coupled limits}
Theorems~\ref{rlowLDPThm} and~\ref{lalowLDPThm} describe the limiting behavior w.r.t.~$r$ and $\la$ separately. Now, we present three results about coupled limits. First if $\la$ and $r$ are such that $r^d\la=\r$ is constant, then by~\cite[Theorem 11.3.III]{daleyPPII2009}, the rescaled Cox process $r^{-1}\vXl$ converges weakly to a homogeneous Poisson point process with intensity $\r$, see Figure~\ref{larFig}. This gives rise to the following statement about the percolation probabilities, where $\bar\theta(\r)$ denotes the continuum percolation probability associated with a homogeneous Poisson point process with intensity $\r$ and connection radius 1.

\begin{figure}[!htpb]
    \includegraphics{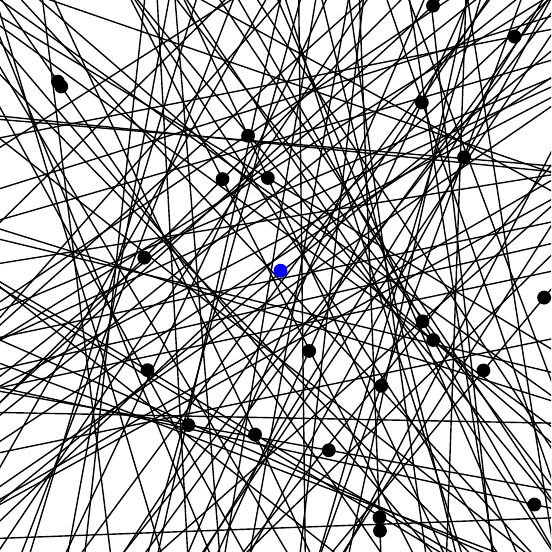}
    \caption{Coupled limit $r \to \infty$ and $\la \tod 0$.}
    \label{larFig}
\end{figure}

%%%%%%%%%%%%%%%%%%%%%%%%%%%%%%%%%%%
%%%%%LARGE R
%%%%%%%%%%%%%%%%%%%%%%%%%%%%%%%%%%%
\begin{theorem}\label{larThm}
Let $\r > 0$. Then, 
    $$\limsup_{\substack{r\uparrow\infty,\, \la\tod 0 \\ \la r^d = \r}}\theta(\la,r) \le \bar\theta(\r).$$
If $\L$ is stabilizing, then
    $$\lim_{\substack{r\uparrow\infty,\, \la\tod 0 \\ \la r^d = \r}}\theta(\la,r) = \bar\theta(\r).$$
\end{theorem}

%%%%%%%%%%%%%%%%%%%%%%%%%%%%%%%%%%%
%%%%%SMALL R -- ABSCONT
%%%%%%%%%%%%%%%%%%%%%%%%%%%%%%%%%%%

For the converse limit $r\tod 0$, $\la\to\infty$ again with $r^d\la=\r$, one cannot hope for such a universal result since the structure of $\L$ becomes prominent also in the limit. In particular, completely different scaling limits emerge for absolutely continuous and singular random intensity measures. 

Let us start with the absolutely continuous case where $\L(\d x) = \ell_x\d x$ with $\ell = \{\ell_x\}_{x \in \R^d}$ a stationary, non-negative random field as in Example~\ref{Ex1}. In this case, $\L^*(\d x) = \ell_x^* \d x$, where $\ell^*$ is the $\ell_0$-size-biased version of $\ell$. Since $\E[\ell_o]=1$, that is,  
\begin{align*}
    \E[f(\ell^*)] =\E[\ell_of(\ell)]
\end{align*}
where $f: [0,\infty)^{\R^d} \too [0, \infty)$ is any bounded measurable function.  Let 
$$L_{\ge} = \{x\in\R^d:\, \ell_x^* \ge \bar\la_c/\r\}$$
denote the superlevel set of $\ell^*$ at level $\bar\la_c/\r$ where $\bar\la_c$ is the critical intensity of the classical Poisson-Boolean model associated to $\bar\theta$. The strict superlevel set $L_>$ is defined accordingly. Then, similarly to the setting analyzed in~\cite{adhoc}, the percolation probability is asymptotically governed by a local and a global constraint, see the left side of Figure~\ref{rlaFig}. Locally, the connected component must leave a small neighborhood around the origin, an event with probability $\theta(\r \ell_o^*)$. Globally, it must be possible to reach infinity along a path of super-critical intensity in $L_{\ge}$, which in the following we denote as $o \leftrightsquigarrow_\L \infty$.

The next result describes the asymptotic behavior of the percolation probability under the assumption that $L_\ge$ is highly connected.

\begin{theorem}\label{rlaAbsThm}
    Let $\r > 0$ and $\ell$ be upper semicontinuous. If, with probability 1, 
    \begin{enumerate}
        \item $\ell_o^* \ne \la_c/\r$,
        \item $\ell^*$ is continuous at $o$, and
        \item the intersection of any connected component of $L_{\ge}$ with $L_>$ remains connected,
    \end{enumerate}
    then
    $$\limsup_{\substack{\la\uparrow\infty,\, r\downarrow 0 \\  \la r^d = \r}}\theta(\la,r) \le \E\big[\bar\theta(\r \ell_o^*) \one\{o \leftrightsquigarrow_\L \infty\}\big].$$
    If, additionally, $\L$ is stabilizing and with a probability tending to 1 in $n$, the set 
    $L_> \cap Q_{5n} \setminus Q_{3n}$
    contains a compact interface separating $\partial Q_{3n}$ from $\infty$.    Then,
    $$\lim_{\substack{\la\uparrow\infty,\, r\downarrow 0 \\  \la r^d = \r}}\theta(\la,r)=\E\big[\bar\theta(\r \ell_o^*)\one\{o\leftrightsquigarrow_\L\infty\}\big].$$
\end{theorem}
In Example~\ref{Ex1}, for the Poisson-Boolean model, the upper semicontinuity is satisfied for $\la_1\ge\la_2$. Further, assumption (1) -- (3) are satisfied as long as $\la_c/\r\ne \la_1,\la_2$. The additional assumption on the existence of the interface can only be guaranteed for sufficiently high intensity of the underlying Poisson process, \cite{cPerc}. For the shot-noise field similar sufficient conditions can be formulated. 
%%%%%%%%%%%%%%%%%%%%%%%%%%%%%%%%%%%
%%%%%SMALL R -- SING
%%%%%%%%%%%%%%%%%%%%%%%%%%%%%%%%%%%

\medskip
Next, we consider a singular setting as in Example~\ref{Ex2}. The scaling relation in Theorem~\ref{rlaAbsThm} was chosen in such a way, that the expected number of neighboring Cox points remains constant. 
If we were to apply this scaling also in the singular case, then the scaling limit would be trivial. Indeed, with high probability on the majority of all edges some subsequent Cox points are separated by gaps of size at least $r$, so that no percolation can occur. 
This is not a problem in case of absolutely continuous measures, since the continuous support of the underlying random intensity measure allows for percolation in all directions. 
Hence, we consider the more appropriate limit where the expected number of gaps per edge remains constant, see the right side of Figure~\ref{rlaFig}.

In this regime, the limiting behavior is governed by an inhomogeneous Bernoulli bond percolation model on the Palm version $S^*$ for the segment system, where the probability for an edge of length $l$ to be open is given by $b^l$ for a suitable $b>0$. For the Poisson-Delaunay tessellation, a homogeneous version was considered for example in~\cite{Ha00}. We write $\theta_{\ms{Ber}}(b)$ for the resulting percolation probability. The next result makes this precise under the assumption that the expected number of gaps per edge is small.
\begin{theorem}\label{rlaSingThm}
			Let $ c> 0$. Then,
				$$\limsup_{\substack{\la \to \infty,\, r \tod 0\\ \la\exp(-\la r) = c}}\theta(\la, r) \le \theta_{\ms{Ber}}(\exp(-c)).$$
                If $\L$ is essentially asymptotically connected and $c$ is sufficiently small, then
				$$\lim_{\substack{\la \to \infty,\, r \tod 0\\ \la\exp(-\la r) = c}}\theta(\la, r) = \theta_{\ms{Ber}}(\exp(-c)).$$
\end{theorem}
The rest of this paper is structured as follows. First in the remainder of the present section, we outline the proofs of the main results. Next, in Section~\ref{Ex}, we provide examples. In Section~\ref{Sim}, we present numerical simulations for the percolation probability. Section~\ref{PhaseTransition} contains the proofs for non-trivial phase transitions. Finally, Section~\ref{limitSec1} is devoted to the large-radius and high-density limit and in Section~\ref{limitSec2} we deal with coupled limits. 

%%%%%%%%%%%%%%%%%%%%%%%%%%%%%%%%%%%%%%%%%%%
%%%%%%%%%%%%%%%%%%%%%%%%%%%%%%%%%%%%%%%%%%%
%%%%%%OUTLINE 
%%%%%%%%%%%%%%%%%%%%%%%%%%%%%%%%%%%%%%%%%%%
%%%%%%%%%%%%%%%%%%%%%%%%%%%%%%%%%%%%%%%%%%%
\subsection{Outline of proofs}
\label{outlineSec}

\subsubsection{Phase transition}
The proof is based on a renormalization argument. More precisely, the stabilization condition makes it possible to create a suitable $b$-dependent auxiliary percolation process, which in turn can be analyzed using the techniques from~\cite{domProd}. The existence of a sub-critical phase is easier to establish, since it suffices to create large regions without any points. For the super-critical regime, more care must be taken in order to produce appropriate connected components. It is at this point that we use the assumption on the asymptotic essential connectedness.

\subsubsection{Large-radius and high-density limits}
\label{hdSec}
For the lower bounds in Theorem~\ref{rlowLDPThm} and~\ref{lalowLDPThm} we consider isolation probabilities of $r$-connected sets containing the origin. If $A \in \mc{R}_r$ is such that $\vXls \cap \partial_r A = \es$, then the points in $A$ are contained in a different connected component of $\vG(\vXls )$ than the points in $\R^d \setminus (A \cup \partial_r A)$. In particular, 
\begin{align}
    \label{isoBoundEq}
    1 - \theta(\la, r) \ge \P(\vXls \cap \partial_r A = \es) = \E[\exp(-\la \L^*(\partial_r A))].
\end{align}
For any $\delta > 0$, this expression is at least
\begin{align*}
    \P(\Lambda^*(\partial_r A) \le \delta + \essinf\Lambda^*(\partial_r A)) \exp(-\lambda(\delta + \essinf\Lambda^*(\partial_r A))),
\end{align*}
which gives the lower bound in Theorem~\ref{lalowLDPThm}. On the other hand, choosing $A = \{o\}$ in~\eqref{isoBoundEq}, taking the logarithm, dividing by $r^d$ and sending $r \to \infty$ gives the lower bound in Theorem~\ref{rlowLDPThm}. The upper bounds in both Theorem~\ref{rlowLDPThm} and~\ref{lalowLDPThm} follow from a Peierls argument in Section~\ref{limitSec1}.

\subsubsection{Coupled limits}\label{Meta}
The proofs of Theorems~\ref{larThm},~\ref{rlaAbsThm} and~\ref{rlaSingThm} are all based on two meta-results. The upper bounds rely on convergence in finite domains. To make this precise, we say that $o \leftrightsquigarrow \partial Q_a$ in $\vG(\vXls)$ if the connected component of $o$ in $\vG(\vXls)$ is not contained in $Q_{a - 2r}$. Moreover, we put 
\begin{align}\label{FinPercProb}
   \theta_a = \P(o \leftrightsquigarrow \partial Q_a).
\end{align}
Now, convergence in finite domains provides the desired upper bounds by virtue of the following elementary upper-semicontinuity result.

    \begin{proposition}[Upper bound via convergence in bounded domains]
        \label{metaProp1}
         Let  $\lambda: [0,\infty) \too [0, \infty)$ be monotone and $s: [0,\infty) \too [1, \infty)$ an increasing function. Let $r_\infty \in \{0, \infty\}$ and $\{a_K\}_{K > 0}$ be a decreasing function such that for every $K > 0$,
        $$\lim_{r \too r_\infty} \theta_{Ks(r)}(\la(r), r) = a_K.$$
        Then,
        $$\limsup_{r \too r_\infty} \theta(\la(r), r) \le \lim_{K \to \infty} a_K.$$
    \end{proposition}
For the lower bound, we use a tightness condition based on a renormalized  percolation process.

    \begin{proposition}[Lower bound via tightness]
        \label{metaProp2}
         Let $\la, r_\infty, s$ and $\{a_K\}_{K > 0}$ be as in Proposition~\ref{metaProp1} and assume that $\L$ is stabilizing. Let $E_r$ denote the event that
         \begin{enumerate}
             \item  $\vG(X^\la)\cap Q_{5s(r)}$ contains a unique connected component intersecting both $\partial Q_{3s(r)}$ and $\partial Q_{5s(r)}$, and 
             \item this component also intersects $Q_{s(r)}$.
         \end{enumerate}
        There exists a constant $q_d \in (0,1)$, only depending on the dimension, such that if for all sufficiently large $r$,
         \begin{align}
             \label{metaCond}
             \min\{\P(\sup_{y\in Q_{5s(r)} \cap \Q^d}R_y < s(r)),  \P(E_r)\} > q_d,
         \end{align}
          then
        $$\lim_{r \too r_\infty} \theta(\la(r), r) = \lim_{K \to \infty} a_K.$$
    \end{proposition}

\section{Examples}\label{Ex}
\subsection{Stabilization and asymptotic essential connectedness}
Let us discuss the case of Poisson-Voronoi tessellations in order to show that the assumptions of stabilization and asymptotic essential connectedness are indeed satisfied by a large class of Cox processes. The case of Poisson-Delaunay tessellations can be dealt with in a similar fashion. 
\begin{example}
    \label{vorEx}
    Let $\L(\d x)=\nu_1(S\cap \d x)$ where $S = h(\vXs)$ is the Poisson-Voronoi tessellation based on the homogeneous Poisson point process $\vXs$ with intensity $\vls$. More precisely, to each $\vXsi \in \vXs$ we associate the cell 
    $$\Xi_i = \{x \in \R^d:\, |x - \vXsi| = \dist(x, \vXs)\}$$
    consisting of all points in $\R^d$ having $X_{\ms{S}, i}$ as the closest neighbor. Then, $S$ is the union of the one-dimensional facets of the collection of cells $\{\Xi_i\}_{i \ge 1}$.

\medskip
    We claim that the random measure in Example~\ref{vorEx} is asymptotically essentially connected. Let us start by verifying the stabilization. We define the radius of stabilization 
 $$
R_x = \inf\{|\vXsi-x|:\, \vXsi \in \vXs\}.
$$ 
    to be the nearest neighbor distance in the underlying Poisson point process. 
Let us check the conditions. 
     First, by stationarity of $\vXs$, also $(h(\vXs),R)$ is translation invariant. Second, let $Q^{(1)}, Q^{(2)}, \ldots, Q^{(d^d)}$ be a sub-division of $Q_n$ into congruent sub-cubes of side length $n/d$. Then, 
\begin{align*}
    1 - \P(\sup_{y\in Q_n\cap\, \Q^d}R_y < n) \le \sum_{i = 1}^{d^d}\P(Q^{(i)}\cap \vXs = \es ) \le d^d \P(Q^{(1)}\cap \vXs = \es) = d^de^{-\vls (n/d)^{d}}
\end{align*}
where the right hand side tends to 0 exponentially fast as $n \to \infty$. This property is referred to in the literature as exponential stabilization, see \cite{PeWa08}. Finally, for almost all realizations of $S$ and for all $x\in S\cap Q_{n}$, by the definition of the Poisson-Voronoi tessellation, there exist unique points $X_{\ms S,1},\dots, X_{\ms S,m}\in\vXs$ such that $|X_{\ms S,1}-x|=\cdots=|X_{\ms S,m}-x|=\inf\{|\vXsi-x|:\, \vXsi\in\vXs\}$. The number of such points depends on the dimension of the facet of $S$ containing $x$. Under the event $\sup_{y\in Q_n\cap\, \Q^d}R_y < n$ we thus have $|X_{\ms S,1}-x|=\cdots=|X_{\ms S,m}-x| \le n$.
In particular, a change in the configuration $\vXs\cap \R^d\sm Q_n(x)$ leaves $x$ unaffected and thus $S\cap Q_n$ is independent of $\vXs\cap \R^d\sm Q_{3n/2}$. Since also the event $\sup_{y\in Q_n\cap\,\Q^d}R_y < n$ is independent of $\vXs\cap\R^d\sm Q_{3n/2}$, the last condition is verified.

\medskip
The Poisson-Voronoi tessellation is also asymptotically essentially connected, which we show now. In order to show the non-emptyness of the support, note that emptyness of the support, i.e. $S\cap Q_n=\es$, implies that $Q_n$ is contained in the Voronoi cell $\X_i$ of some $\vXsi\in\vXs$. Choose any two points $x,y\in Q_n\subset \X_i$ with $n<|x-y|$. Then, under the event $\sup_{z\in Q_{n}\cap\,\Q^d}R_z < n/2$, the points $x,y$ must have some distinct Poisson points $X, Y\in\vXs$ with $|x-X|<n/2$ and $|y-Y|<n/2$ and hence, 
\begin{align}\label{Contra}
n<|x-y|\le |x-\vXsi|+|y-\vXsi|\le  |x-X|+|y-Y|<n, 
\end{align}
a contradiction. As for the connectedness, denote by $S_1,\dots, S_k$ the connected components in $S\cap Q_n$. By the definition of the Voronoi tessellation, every void space $V$, which separates two of the connected components in $Q_n$, must be the intersection of one Voronoi cell $\X$ with $Q_n$.
Let $\partial\Xi$ denote the boundary of $\Xi$. 
We claim that, under the event $\sup_{y\in Q_{2n}\cap\, \Q^d}R_y < n/2$, we have that $\partial\X\cap Q_n$ is connected in $Q_{2n}$. Indeed, let $\vXsi\in\vXs$ be such that $\Xi=\Xi_i$ then, since $\partial\Xi_i$ contains points in $Q_n$, we have $\vXsi\in Q_{2n}$. 
If $\vXsi\in Q_n$ then $\Xi_i\subset Q_{2n}$ since for all $x\in \X_i$ we have $|x-\vXsi|<n/2$ which implies that $x\in Q_{2n}$. Hence, in this case the associated disconnected components in $Q_n$ must be connected in $Q_{2n}$. On the other hand, if $\vXsi\in Q_{2n}\sm Q_n$ then there exists a chord in $\Xi_i$ starting at $\vXsi$ and crossing $Q_n$ completely. But any such chord has maximal length $n/2$, and thus again the associated disconnected components in $Q_n$ must be connected in $Q_{2n}$.
Since the argument holds for any void space $V$, we have connectedness of all $S_1,\dots, S_k$ in the larger volume $Q_{2n}$.
\end{example}

\subsection{Computation of the rate function in Theorem~\ref{rlowLDPThm} for the shot-noise field}
Let $\L$ be the shot-noise field with $\ell_x=\sum_{X_i\in \vXs}k(x-X_i)$ for some non-negative integrable $k$ with compact support and $\vXs$ be a Poisson point process with intensity $\vls>0$.
Then, for $K=\int k(x)\d x$ we can calculate, 
    \begin{align*}
        \log \E[\exp(-\la\L(Q_r))]&=\log \E[\exp(-\la \int_{Q_r}\sum_{X_i\in\vXs}k(x-X_i)\d x)]\\
        &=\vls\int  \, (e^{-\la \int_{Q_r}\,  k(x-y)\d x}-1)\d y\\
        &=\vls\int  \, (e^{-\la \int_{Q_r(y)}\,  k(x)\d x}-1)\d y.
    \end{align*}
    By separating the domain of integration into $Q_r$ and $\R^d\setminus Q_r$, we arrive at
    \begin{align*}
        r^d\vls(e^{-\la K}-1)+\vls\int_{Q_r} \, (e^{-\la \int_{Q_r(y)}\,  k(x)\d x}-e^{-\la K})\d y +\vls\int_{\R^d\sm Q_r}  \, (e^{-\la \int_{Q_r(y)}k(x)\d x}-1)\d y.
    \end{align*}
We claim that the second and third summand in the last line are of order $o(r^d)$. Indeed, let $r_0>0$ be large enough such that the support of $k$ is contained in $Q_{r_0}$. Then, for the second summand, we have for $r> r_0$ that
    \begin{align*}
       0\le  \int_{Q_r} \, (e^{-\la \int_{Q_r(y)}k(x)\d x}-e^{-\la K})\d y &=\int_{Q_{r-r_0}}  \, (e^{-\la \int_{Q_r(y)}k(x)\d x}-e^{-\la K})\d y\\
        &\qquad+\int_{Q_r\sm Q_{r-r_0}}  \, (e^{-\la \int_{Q_r(y)}k(x)\d x}-e^{-\la K})\d y\\
 &\le |Q_r\sm Q_{r-r_0}|.
    \end{align*}
For the third summand, using similar arguments, 
        \begin{align*}
      0\ge  \int_{\R^d\sm Q_r} \, (e^{-\la \int_{Q_r(y)}\,  k(x)\d x}-1)\d y =\int_{Q_{r+r_0}\sm Q_r}  \, (e^{-\la \int_{Q_r(y)}\,  k(x)\d x}-1)\d y\ge -|Q_{r+r_0}\sm Q_{r}|.
    \end{align*}

\section{Simulations}\label{Sim}
\subsection{Simulations}
In order to provide numerical illustrations for the main mathematical theorems, we estimated the actual percolation probability for a variety of parameters via Monte Carlo simulations. More precisely, $\Lambda$ is assumed to be the random measure given by the edge length of a planar tessellation as in Example~\ref{Ex2}. Here, we consider either Poisson-Voronoi tessellation or Poisson-Delaunay tessellation, and fix the length intensity $\E[\Lambda(Q_1)] = 20$. 
\begin{figure}[!htbp]
    \includegraphics[width=.31\textwidth]{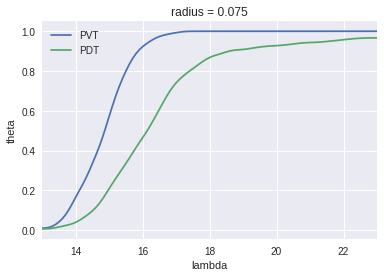}
        \includegraphics[width=.31\textwidth]{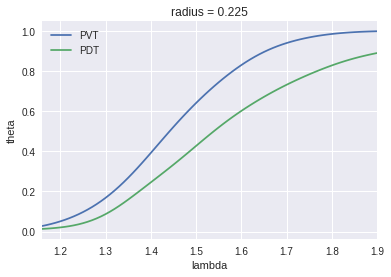}
        \includegraphics[width=.31\textwidth]{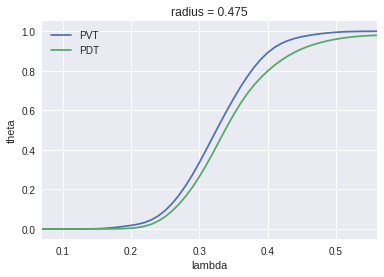}
                    \caption{Percolation probability as a function of intensity for different radii and random intensity measures.}
                                    \label{simFig}
\end{figure}

In Figure~\ref{simFig}, we present the estimated percolation probability $\theta(\la, r)$ as a function of the parameter $\la$ for three choices of the radii: $r = 0.075$, $r = 0.225$ and $r = 0.475$. In Theorem~\ref{larThm} we have seen that in the asymptotic setting of a large radius and small intensity, the percolation probability does not depend on the choice of the random intensity measure. It converges to the percolation probability of the Poisson-Boolean model. This behavior is reflected in the right-most panel of Figure~\ref{simFig} where $r = 0.475$, as there is very little difference between the percolation probability in the Voronoi or Delaunay setting. For $r = 0.225$, we see that the geometry of the random intensity measure influences substantially the percolation probability. This is even more prominent for smaller radii, such as $r = 0.075$. Indeed, Theorem~\ref{rlaSingThm} describes the behavior for small radii, also in the asymptotic regime, but here the dependence of the percolation probability on the underlying random intensity measure is not lost in the limit.

\section{Proof of phase transitions}\label{PhaseTransition}
The main idea is to introduce a renormalization scheme reducing the continuum percolation problem to a dependent lattice percolation problem. To make this work, we rely crucially on the stabilization assumption. It allows us to make use of the standard $b$-dependent percolation arguments presented in~\cite[Theorem 0.0]{domProd}.
\subsection{Existence of sub-critical phase}
In the renormalization we single out large regions that do not contain any Cox points and where one has good control over the spatial dependencies induced by $\L$. In the following, we put $R(Q_n(x)) = \sup_{y\in Q_n(x)\cap\,  \Q^d}R_y$, where $R_y$ is the stabilization radius associated to $\L$.
\begin{proof}[Proof of Theorem~\ref{subCritStabThm}]
A site $z \in \Z^d$ is \emph{$n$-good} if 
	\begin{enumerate}
		\item $R(Q_{n}(nz))<n$ and
		\item $\vXl \cap Q_n(nz) = \es$
	\end{enumerate}
	
	A site $z \in \Z^d$ is \emph{$n$-bad} if it is not $n$-good.
	By property (2), percolation of the Gilbert graph implies percolation of $n$-bad sites. Hence, it suffices to verify that $n$-bad sites do not percolate if $\la$ is sufficiently small. 

The process of $n$-bad sites is $3$-dependent as can be seen from the definition of stabilization. Moreover, 
\begin{align*}
\limsup_{n\uparrow\infty}\limsup_{\la\downarrow0}\P(z\text{ is }n\text{-bad})
&\le\limsup_{n\uparrow\infty}\limsup_{\la\downarrow0}\big(\P(R(Q_{n})\ge n) + 1-\E(e^{-\la\L(Q_{n})})\big)=0.
\end{align*}	
By~\cite[Theorem 0.0]{domProd}, we conclude that the process of $n$-bad sites is stochastically dominated by a sub-critical Bernoulli percolation process. In particular, with probability one, there is no infinite path of $n$-bad sites.
\end{proof}

\subsection{Existence of super-critical phase}
This time, our goal is to identify large regions where the support of $\L$ is well-connected and the Cox points are densely distributed on the support of $\L$ in these regions. 
\begin{proof}[Proof of Theorem~\ref{supCritThm}]
A site $z \in \Z^d$ is \emph{$n$-good} if 
	\begin{enumerate}
		\item $R(Q_{6n}(nz))<n/2$, 
		\item $\vXl \cap Q_n(nz) \ne \es$ and 
		\item every $X_i, X_j \in \vXl \cap Q_{3n}(nz)$ are connected by a path in $\vG(\vXl) \cap Q_{6n}(nz)$.
	\end{enumerate}
    We claim that if there exists an infinite connected component $\mc{C}$ of $n$-good sites, then $\vG(\vXl)$ percolates. Indeed, 
    let $z,z'\in \mc{C}$ and $X_i, X_{i'} \in \vXl$ be such that $X_i\in Q_n(nz)$, $X_{i'}\in Q_n(nz')$. If $\{z, z_1, \ldots,z'\} \subset \mc{C}$ is any finite path connecting $z$ and $z'$, by property (2) we can choose points $X_{j} \in \vXl \cap Q_n(nz_{j})$ for every $j\ge1$. Using property (3), we see that these points as well as $X_i$ and $X_{i'}$ are contained in a connected component in $\vG(\vXl)$. This gives the existence of an infinite cluster. 

It remains to show that the process of $n$-good sites is supercritical for sufficiently large $\la$.
The process of $n$-good sites is $7$-dependent as can be seen from the definition of stabilization. Moreover, writing $A,B,C$ for the events (1), (2) and (3) we have 
	\begin{align*}
\limsup_{n\uparrow\infty}\limsup_{\la\uparrow\infty}\P(z\text{ is }n\text{-bad})
&=\limsup_{n\uparrow\infty}\limsup_{\la\uparrow\infty}\P(A^c\cup B^c\cup C^c)\\
        &\le\limsup_{n\uparrow\infty}\limsup_{\la\uparrow\infty}\big(\P(A^c) + \P(B^c\cap A)+\P(C^c\cap A)\big)
\end{align*}
    where $\limsup_{n\uparrow\infty}\P(A^c)\le 12^d\limsup_{n\uparrow\infty}\P(R(Q_{n/2})\ge n/2) = 0$ by stabilization. Further, by condition (1) in Definition~\ref{aecDef} and dominated convergence,
\begin{align*}
    \lim_{\la\uparrow\infty}\P(B^c\cap A)&\le \lim_{\la\uparrow\infty}\E[e^{-\la\L(Q_n)}\one\{R(Q_{2n})<n/2\}]\le\lim_{\la\uparrow\infty}\E[e^{-\la\L(Q_n)}\one\{\L(Q_{n})>0\}] = 0.
\end{align*}
 Finally, by asymptotic essential connectedness, if $R(Q_{6n})<n/2$, then there exists a connected component $\D$ with $\supp(\L_{Q_{3n}})\subset\D\subset\supp(\L_{Q_{6n}})$. Note that the support of $\L$ will be filled with Cox points in the $\la\uparrow\infty$ limit and thus any two points in $\supp(\L_{Q_{3n}})$ must be connected eventually. More precisely, let $(Y_i)_{1 \le i \le N}$ be a finite set of points in $\D$ such that 
 $\D\subset\bigcup_{1\le i\le N}\D_i$ where $\D_i=B_{r/4}(Y_i)$.
Further, let
$$D = \{\vXl\cap \D_i\neq\emptyset\text{ for all }1 \le i \le N\}$$ 
    denote the event that every $\D_i$ contains at least one Cox point. Then, $D\cap A \subset C\cap A$ yields that $\P(C^c\cap A)\le \P(D^c\cap A)$.
In addition, by asymptotic essential connectedness,
\begin{align*}
    \P(D^c\cap A) \le \sum_{1 \le i \le N}\E[e^{-\la\L(\D_i)}\one\{R(Q_{6n})<n/2\}] \le \sum_{1 \le i \le N}\E[e^{-\la\L(\D_i)}\one\{\L(\D_i)>0\}].
\end{align*}
Now using dominated convergence,
\begin{align*}
\lim_{\la\uparrow\infty}\sum_{1 \le i \le N}\E[e^{-\la\L(\D_i)}\one\{\L(\D_i)>0\}] = 0,
\end{align*}
as asserted.

Again by~\cite[Theorem 0.0]{domProd}, the process of $n$-good sites is stochastically dominated from below by a super-critical Bernoulli percolation process, as required.
\end{proof}

\section{Proofs of Theorems~\ref{rlowLDPThm} and~\ref{lalowLDPThm}}
\label{limitSec1}
\subsection{Proof of Theorem~\ref{rlowLDPThm}}\label{CoxPercProof}
As explained in Section~\ref{hdSec}, we only need to prove the upper bound together with the existence of the limiting cumulant generating function $I^*$. We prove the assertion for finite domains and then rely on a Peierls argument to establish that the $\{\theta_{Kr}\}_{K \ge 1}$ (as defined in~\eqref{FinPercProb}), $K \ge 1$ form exponentially good approximations of $\theta$ in the sense of~\cite[Definition 4.2.10]{dz98}. We put $\k_d=|B_1(o)|$. To prove Theorem~\ref{rlowLDPThm}, we proceed in two steps:
\begin{enumerate}
    \item for every $\la > 0$ and $K \ge 1$,
        $$\limsup_{r \to \infty} r^{-d} \log (1 - \theta_{Kr}(\la, r)) \le -\kappa_dI^*,$$
    \item for every $\la > 0$ and all sufficiently large $K \ge 1$,
        $$\limsup_{r \to \infty} r^{-d} \log (\theta_{Kr}(\la, r) - \theta(\la, r)) < -\kappa_dI^*.$$
\end{enumerate}
%%%OUTLINE
The idea for the upper bound on finite domains is to consider the convex hull of the cluster at the origin. In particular, there are no Cox points in a forbidden volume formed by all points
outside the convex hull but within distance at most $r$ of one of its vertices. By Steiner's formula from convex geometry~\cite[Theorem 1.1]{steiner}, the volume of this set is at least $\kappa_dr^d$, so that we arrive at the desired upper bound for finite $K$.

To prove the exponentially good approximation, note that if the cluster at the origin is finite but percolates beyond $Q_{Kr}$, then we can define a substantially larger forbidden volume, giving rise to a faster exponential decay.

%%%LOWER BOUND
The main ingredient in the proof is a large deviation formula for the Laplace transform of the random measure $\L^*$ in a large set.
\begin{lemma}
	\label{concLenLem}
  Let $\la > 0$ and assume that $\L$ is $b$-dependent and $\L(Q_1)$ has all exponential moments.
    \begin{enumerate}
        \item Then the following limit exists
            $$I^* = -\lim_{r \to \infty} r^{-d} \log\E[\exp(-\la\L(Q_r))].$$
        \item Let $A \subset \R^d$ be compact with $|\partial A| = 0$ and $A \ne \partial A$. Then,
    $$\lim_{r \to \infty} r^{-d} \log\E[\exp(-\la \L^*(rA))] = -|A| I^*.$$
    \end{enumerate}
\end{lemma}

The proof of Lemma~\ref{concLenLem} reveals that the statement remains true if $\L^*$ is replaced by $\L$. 
Before proving Lemma~\ref{concLenLem}, we explain how it is used to establish Theorem~\ref{rlowLDPThm}.  For this, we need two deterministic Peierls-type results. For a locally finite set $\varphi \subset \R^d$ and $x \in \varphi$ let 
$\C_r(x) = \C_{r}(x, \varphi)$ denote the vertices in connected component of $\vG(\varphi)$ containing $x \in \varphi$. Moreover, let 
    $$\C_{r, \varepsilon}(x) = \C_{r, \varepsilon}(x, \varphi)= \{r\e z \in r\e \Z^d:\,   Q_{r\e}(r\e z)\cap  \C_r(x) \ne \es\}$$
    denote the family of sites whose associated $r\e$-cube contains a point of $\C_r(x)$. Finally, for any $\e > 0$ and $A \subset \R^d$ we let
    $$A^{\oplus \e} = \bigcup_{a \in A} Q_\e(a)$$
    denote the \emph{dilation} of $A$ by $Q_\e$.

%%%%%%%%%%%%%%%
%%%PEIERLS
%%%%%%%%%%%%%%%

\begin{lemma}
    \label{peierlsLem1}
    Let $\e, r  > 0$, $\varphi \subset \R^d$ be locally finite and $x \in \varphi$. Then, $\C_{r, \varepsilon}(x)$ is an $r(1 + d\e)$-connected subset of $r\e\Z^d$ and 
    $$\varphi \cap \partial_{r(1 - d\e)} \C_{r, \e}^{\oplus r \e}(x) = \es.$$
\end{lemma}
\begin{proof}
    The discretized cluster $\C_{r, \e}(x)$ is $r(1 + d\e)$-connected, since for any $z, z' \in \Z^d$ the existence of $y \in Q_{r\e}(r\e z)$ and $y' \in Q_{r\e}(r\e z')$ with $|y - y'| < r$ implies that $r\e |z - z'| \le r(1 + d\e)$. Moreover, for $y \in \varphi$ with $\dist(y, \C_{r, \e}^{\oplus r\e}(x)) \le r(1 - d\e)$ we have $\dist(y, \C_r(x)) \le r$, so that $y \in \C_r(x)$.
\end{proof}

Next, we let $\Ce(x)$ denote the \emph{external boundary} of the cluster $\C_r'(x) = \bigcup_{y \in \C_r(x)} B_{r/2}(y)$. That is, all points on $\partial \C_r'(x)$ connected to $\infty$ by a continuous path in $\R^d \setminus \C'_r(x)$. Moreover, for $\e > 0$, we let 
$$\Cee(x) = \{r\e z \in r\e \Z^d:\, Q_{r\e}(r\e z) \subset \C_r'(x)  \text{ and } B_{2dr\e}(r\e z) \cap \Ce(x)  \ne \es \}$$
denote the discretization of the external boundary.
\begin{lemma}
    \label{peierlsLem2}
    Let $r > 0$, $\e \in (0, 1/(2d))$ and $\varphi \subset \R^d$ be locally finite. Then, $\Cee(x)$ is $5dr\e$-connected.
\end{lemma}
\begin{proof}
    Let $r\e z, r\e z' \in \Cee(x)$ be arbitrary and choose $y, y' \in \Ce(x)$ such that $\max\{|y - r\e z|, |y' - r\e z'|\} \le 2d r\e$. We prove the claim by induction on $\lfloor \ell/(dr\e)\rfloor$, where $\ell$ is the length of the shortest connecting continuous path $\pi$ between $y$ and $y'$ on $\partial \C_r'(x)$. Now, choose $y''$ on $\pi$ such that the length of the connecting path between $y$ and $y''$ is given by $dr\e$. Furthermore, we assert that there exists $z''\in\Z^d$ such that $r\e z'' \in \Cee(x)$ and that $|y'' - r\e z''| \le 2dr\e$. Indeed, choose $x'\in\varphi$ and $p\in [x',y'']$ such that $y''\in\partial B_{r/2}(x')$ and $|p-y''|=dr\e$. Then, any $z''\in\Z^d$ satisfying $p\in Q_{r\e}(r\e z'')$ has the desired properties. This implies that, 
    $$|r\e z'' - r\e z| \le  |r\e z'' - y''| + dr\e + |r\e z - y| \le 5dr\e,$$
    so that using the induction hypothesis on $z''$ and $z'$ concludes the proof.
\end{proof}

%%%%%%%%%%%%%%%
%%%UPPER BOUND
%%%%%%%%%%%%%%%
\begin{proof}[Proof of Theorem~\ref{rlowLDPThm}, upper bound for fixed $K$]
We apply Lemma~\ref{peierlsLem1} with $\varphi = r^{-1}\vXls$ and $r = 1$ to obtain that
\begin{align*}
1-\theta_{Kr}(\la, r)\le \sum\P(r^{-1}X^{*, \la} \cap \partial_{1 - d \e} A^{\oplus \e} = \es),
\end{align*}
where the sum is over all $(1 + d \e)$-connected subsets $A$ of $Q_K\cap\e\Z^d$. 
Since this sum is finite, it suffices to consider an arbitrary $A$. 
Since $\vXls$ is a Cox point process with intensity measure $\la\L^*$, Lemma~\ref{concLenLem} reduces the proof to the assertion that the volume of $\partial_{1 - d\e}A^{\oplus \e}$ is at least $\kappa_d (1 - 2d\e)^d$.
         
    Let $H(A)$ denote the convex hull of $A^{\oplus \e}$ and $V(A)$ set of all points outside $H(A)$ having distance at most $(1 - 2d\e)$ from a vertex of $H(A)$. Then, $V(A) \subset \partial_{1 - d\e}A^{\oplus \e}$. Now, the Steiner formula from convex geometry (see~\cite[Theorem 1.1]{steiner} and the proof) implies that $|V(A)| \ge \kappa_d (1 - 2d\e)^d$.
\end{proof}

\begin{proof}[Proof of Theorem~\ref{rlowLDPThm}, upper bound $K \to \infty$]
	%%DECOMPOSE ACCORDING TO CONTOUR SIZE
    By Palm calculus, it suffices to prove an upper bound for the probability that there exists a finite connected component of the Gilbert graph containing a point $x$ in $Q_1$ and percolating beyond $Q_{(K-1)r}$. Now, we apply Lemma~\ref{peierlsLem2} with $\varphi = r^{-1} \vXl$, $r = 1$ and $\e = 1/(5d)$. In particular, the discretized external boundary $\C = \Cee(x)$ forms a $1$-connected set not intersecting $r^{-1}\vXl$.  We claim that every site in $\C$ has $d_\infty$-distance to the origin at most $\#\C$.

    Indeed, assume that $(5d)^{-1} z \in \C$ was a site of $d_\infty$-distance $d_0 \ge 1$ from the origin. Then, there exists a half-space containing the origin and such that every point in that half-space is at least of $d_\infty$-distance $d_0$ to $(5d)^{-1}z$. This half-space hits the external boundary, so that we can choose    another site $(5d)^{-1} z' \in \C$ that is of distance at least $d_0$ from $(5d)^{-1} z$. Then $d_0 \ge \#\C$ since $\C$ is $1$-connected. In particular, writing $\mc{A}_k$ for the family of $1$-connected $k$-element subsets in $(5d)^{-1}\Z^d$ having at most distance $k$ to the origin, we arrive at 
    \begin{align}
        \label{kInfEq}
        \theta_{Kr}(\la, r) - \theta(\la, r) \le \sum_{k \ge (K - 1)/2} \sum_{A \in \mc{A}_k} \P(r^{-1}\vXl \cap A^{\oplus (5d)^{-1}} = \es).
    \end{align}
    Moreover, each $A \in \mc{A}_k$  decomposes into $2^d$ subsets such that at least one of them consists of $\lceil 2^{-d}k\rceil$ or more non-adjacent squares. Hence, by $br^{-1}$-dependence and stationarity of $r^{-1}\vXl$, for large $r$,
    $$\P(r^{-1}\vXl \cap A^{\oplus (5d)^{-1}} = \es) \le \P(r^{-1}\vXl \cap Q_{(5d)^{-1}} = \es)^{2^{-d}k}$$
    
    Note that $\#\mc{A}_k \le (4k)^da^k$ for a suitable $a > 0$~\cite[Lemma 9.3]{penrose}, so that
\begin{align*}
    \theta_{Kr}(\la, r) - \theta(\la, r) &\le \sum_{k \ge (K - 1)/2} (4k)^da^k \P(r^{-1}\vXl \cap Q_{(5d)^{-1}} = \es)^{2^{-d}k} \\
        &\le (2a\P(r^{-1}\vXl \cap Q_{(5d)^{-1}}(o) = \es)^{2^{-d}})^{(K - 1)/2}
    \end{align*}
    for sufficiently large $K$. By Lemma~\ref{concLenLem}, the last expression decays to zero at an arbitrary high exponential rate in $r^d$ if $K$ is chosen sufficiently large.
	\end{proof}

%%%%%%%%%%%%%%%%%%%%%%%%%%%%%%%%
%%%EXPONENTIAL EQUIVALENCE
%%%%%%%%%%%%%%%%%%%%%%%%%%%%%%%%
Now, we prove Lemma~\ref{concLenLem} relying on the LDP for near-additive functionals established in~\cite[Theorem 2.1]{SeYu01}. Since near-additivity requires the existence of a specific coupling that is not immediate in our setting, we reprove a variant of~\cite[Theorem 2.1]{SeYu01} tailored to our needs. To begin with, we observe that the assumption on $b$-dependence provides strong exponential approximation properties.
\begin{lemma}
    \label{markovLem}
    Assume that $\L$ is $b$-dependent and $\L(Q_1)$ has all exponential moments.
Let $A \subset \R^d$ be compact, $\a>0$ and put $s_\alpha = \log\E[\exp(\alpha\L(Q_b))]$ and $n = \#\{z \in \Z^d:\, Q_b(bz) \cap A \ne \es\}$.
    Then,
    $$\log \P(\L(A) \ge L) \le ns_\alpha - 2^{-d} \alpha L + d \log 2.$$
\end{lemma}
\begin{proof}
    As in the proof of Theorem~\ref{rlowLDPThm}, we partition the set 
    $$\{z \in \Z^d:\, Q_b(bz) \cap A \ne \es\}$$
    into $2^d$ subsets $\{B_i\}_{1 \le i \le 2^d}$ of non-adjacent sites. In particular, by Markov's inequality 
    $$\P(\L(A) \ge L) \le \sum_{i = 1}^{2^d} \P(\L(B_i \oplus Q_b) \ge L2^{-d}) \le \sum_{i = 1}^{2^d}\exp( \#B_i s_\alpha - \alpha L2^{-d}) \le 2^d\exp(ns_\alpha - \alpha L2^{-d}),$$
    as required.
\end{proof}

In particular, we deduce that contributions of surface order can be neglected for our considerations.  Then, applying Lemma~\ref{markovLem} with $L = \e r^d$, $\alpha = \e^{-2}$, and $n = o(r^d)$ gives the following result, where for any $A, B \subset \R^d$ we let
    $$A \oplus B = \bigcup_{a \in A} a + B\quad\text{and}\quad A \ominus B = \{a \in A:\, a+B\subset A\}$$
    denote the \emph{dilation}, respectively \emph{erosion}, of $A$ by $B$.
\begin{corollary}
    \label{expApproxLem}
    Let $s > 0$ and $A \subset \R^d$ is a compact set with $|\partial A| = 0$. If  $\L$ is $b$-dependent, then, as $r \to \infty$, the random variable $\L(r\partial A \oplus Q_s)$ is exponentially equivalent to zero.
\end{corollary}

Equipped with these ancillary results, we now proceed with the proof of Lemma~\ref{concLenLem}. 

\begin{proof}[Proof of Lemma~\ref{concLenLem}]
    %%CUBES VIA YS-Theory
    We proceed along the lines of~\cite[Theorem 2.1]{SeYu01}. First, choosing $n=n(r) = \lfloor r \rfloor$, 
    $$\E[\exp(-\la \L(Q_{n + 1}))] \le \E[\exp(-\la \L(Q_r))] \le \E[\exp(-\la \L(Q_{n}))].$$
    Hence, it suffices to show that 
    $$\liminf_{n \to \infty}n^{-d}\log\E[\exp(-\la \L(Q_n))] = -I^*,$$
    where
    $$I^* =- \limsup_{n \to \infty} n^{-d}\log\E[\exp(-\la \L(Q_n))].$$  
    Let $\e > 0$ be arbitrary and choose $k_1 \ge 1$ such that 
    $$k_1^{-d}\log \E[\exp(-\la \L(Q_{k_1}))] \ge -\e - I^*.$$
    Next, we assert that
    \begin{align}
        \label{yukAssertEq}
        \liminf_{m \to \infty} \frac{1}{(k_1 m)^d}\log \E[\exp(-\la \L(Q_{k_1m}))] \ge -2\e -  I^*.
    \end{align}
    Indeed, subdivide $Q_{k_1m}$ into $m^d$ congruent sub-cubes $Q^{(1)}, \ldots, Q^{(m^d)}$ of side length $k_1$. Then, assuming $b \ge 1$,
    the number of cubes $Q_b(bz)$ intersecting $Q^{(i)} \setminus(Q^{(i)} \ominus Q_b)$ is at most $2^dk_1^{d-1}$. 
    In particular, by Lemma~\ref{markovLem}, 
    $$\P\Big(\sum_{1\le i \le m^d}\L(Q^{(i)} \setminus(Q^{(i)} \ominus Q_b)) > \e \la^{-1} (k_1 m)^d\Big) \le \exp(-2(k_1m)^dI^*)$$
    if $k_1$ is large enough such that 
    $$2^dk_1^{d-1}s_{4\cdot 2^dI^*\la/\e}  - 4I^*k_1^d\le -2I^*k_1^d.$$
    Moreover, by $b$-dependence, 
    \begin{align*}
       \log\E\Big[\exp\Big(- \la \sum_{1 \le i \le m^d}\L(Q^{(i)} \ominus Q_b)\Big)\Big]&=m^d\log \E[\exp(- \la \L(Q^{(1)} \ominus Q_b))]\\
       &\ge m^d\log \E[\exp(- \la \L(Q_{k_1}))]\\
       &\ge k_1^dm^d(-I^*-\e).
    \end{align*}
    In particular, considering
    \begin{align*}
        \E[\exp(-\la \L(Q_{k_1m}))] &= \E\Big[\exp\Big(- \la \sum_{1\le i \le m^d}\L(Q^{(i)} \setminus(Q^{(i)} \ominus Q_b)) - \la \sum_{1 \le i \le m^d}\L(Q^{(i)} \ominus Q_b)\Big)\Big]\\
        &\ge 
        \E\Big[\exp\Big(- (k_1m)^d\e - \la \sum_{1 \le i \le m^d}\L(Q^{(i)} \ominus Q_b)\Big)\Big]\\
        &\phantom{=} - \P\Big(\sum_{1\le i \le m^d}\L(Q^{(i)} \setminus(Q^{(i)} \ominus Q_b)) > \e \la^{-1}(k_1 m)^d\Big),
    \end{align*}
    we see that assertion~\eqref{yukAssertEq} follows.

\medskip
    Now, let $A \subset \R^d$ be as in the claim of part (2). Then, by definition of the Palm distribution,
    $$\E[\L(Q_1)\exp(-\la \L(rA \oplus Q_1))]\le \E[\exp(-\la \L^*(rA))] \le \E[\L(Q_1)\exp(-\la \L(rA \ominus Q_1))].$$
Using H\"older's inequality, the right-hand side can be further estimated by 
$$
\E[\L(Q_1)^{(1+\e)/\e}]^{\e/(1+\e)}\E[\exp(-\la\L(rA \ominus Q_1))]^{1/(1+\e)}$$
for some arbitrarily small $\e>0$ where the first factor does not depend on $r$. For the left-hand side, we can further estimate using $b$-dependence
    \begin{align*}
&\E[\L(Q_1)\exp(-\la( \L(rA \oplus Q_1)\sm Q_{1+2b})-\la\e r^d)\one\{\L(Q_{1+2b})< \e r^d\}]\\
&\ge\E[\L(Q_1)\exp(-\la( \L(rA \oplus Q_1)\sm Q_{1+2b})-\la\e r^d)]-\E[\L(Q_1)\one\{\L(Q_{1+2b})\ge \e r^d\}]\\
&\ge\E[\exp(-\la( \L(rA \oplus Q_1))-\la\e r^d)]\E[\L(Q_1)]-\E[\L(Q_1)^2]^{1/2}\P[\L(Q_{1+2b})\ge \e r^d]^{1/2}.
\end{align*}   
The very last factor converges to zero exponentially fast with arbitrary rate by    Lemma~\ref{markovLem}.
    Further, note that $((rA\oplus Q_1)\sm rA) \cup (rA\sm(rA\ominus Q_1)) \subset r\partial A\oplus Q_1$. Hence, by Corollary~\ref{expApproxLem} it suffices to prove the claim of part (2) with $\E[\exp(-\la \L^*(rA))]$ replaced by $\E[\exp(-\la \L(rA))]$. For $\delta > 0$ let 
    $$A_\delta = \bigcup_{z \in \Z^d:\, Q_\delta (\delta z) \cap A \ne \es} Q_\delta (\delta z)$$
    denote the union of all $\delta$-cubes intersecting $A$. For $r\de > b$, Corollary~\ref{expApproxLem} and $b$-dependence yield 
    \begin{align*}
        \liminf_{r \to \infty}r^{-d}\log \E[\exp(-\la \L(rA))]  &\ge \liminf_{r \to \infty}r^{-d}\log \E[\exp(-\la \L(rA_\delta))]\\
        &\ge |A_\delta|\delta^{-d} \liminf_{r \to \infty}r^{-d}\log \E[\exp(-\la \L(Q_{r\delta}))]\\
        &= -|A|I^* - |A_\delta \setminus A| I^*.
    \end{align*}
    In particular, since $|\partial A| = 0$, sending $\delta \tod 0$ gives that 
    $$\liminf_{r \to \infty}r^{-d}\log \E[\exp(-\la \L(rA))] \ge -|A|I^*.$$
    Noting that a  similar argument shows that 
    $$\limsup_{r \to \infty}r^{-d}\log \E[\exp(-\la \L(rA))] \le -|A|I^*,$$
    concludes the proof.
\end{proof}

\subsection{Proof of Theorem~\ref{lalowLDPThm}}
Again, as explained in Section~\ref{hdSec}, we only need to prove the upper bound. To achieve this goal, we proceed as in the proof of Theorem~\ref{rlowLDPThm}. Let $A$ be an arbitrary $(r + dr\varepsilon)$-connected subset of $ Q_{Kr} \cap r\varepsilon\Z^d$, then
\begin{align*}
\limsup_{\la \to \infty}\la^{-1} \log \P(X^{*, \la} \cap \partial_{r(1 - d \e)} A^{\oplus r \e} = \es)&=\limsup_{\la \to \infty}\la^{-1} \log \E[\exp(-\la\L^*(\partial_{r(1 - d \e)} A^{\oplus r \e}))]\\
& \le -\essinf\Lambda^*(\partial_{r(1 - d \e)} A^{\oplus r\e}).
\end{align*}
Since the number of possible choices for $A$ is finite, we conclude as in Section~\ref{CoxPercProof}, that
$$\limsup_{\la\uparrow\infty}\la^{-1}\log(1-\theta_{Kr}(\la, r))\le -\lim_{\e \tod 0}\inf_{A \in \mc{R}_{r + \e}}\essinf \L^*(\partial_{r - \varepsilon} A).$$

It remains to bound $\theta_{Kr}(\lambda, r) - \theta(\lambda, r)$. Again, we proceed as in the proof of Theorem~\ref{rlowLDPThm} to arrive at 
$$\theta_{Kr}(\lambda, r) - \theta(\lambda, r) \le (2a\P(\vXl \cap Q_{r(5d)^{-1}} = \es)^{2^{-d}})^K =  (2a\E[\exp(-\la \L(Q_{r(5d)^{-1}}))]^{2^{-d}})^K$$
for sufficiently large $K$.
Since, by assumption, $\essinf \L(Q_{r(5d)^{-1}}) > 0$, the last expression again decays to zero at an arbitrary high exponential rate in $\la$ if $K$ is chosen sufficiently large.

\section{Proofs of Theorems~\ref{larThm},~\ref{rlaAbsThm} and~\ref{rlaSingThm}}
\label{limitSec2}
To begin with, we prove the meta results announced in Section~\ref{Meta}.
\begin{proof}[Proof of Proposition~\ref{metaProp1}]
        First, for any $K_0 > 0$,
        $$\limsup_{r \too r_\infty} \lim_{K \to \infty} \theta_{K s(r)}(\la(r), r) \le \lim_{r \too r_\infty} \theta_{K_0 s(r)}(\la(r), r) = a_{K_0},$$
    so that
    $$\limsup_{r \too r_\infty}\theta(\la(r), r) =\limsup_{r \too r_\infty} \lim_{K \to \infty} \theta_{K s(r)}(\la(r), r) \le \lim_{K \to \infty}a_K$$
    as required.
\end{proof}
\begin{proof}[Proof of Proposition~\ref{metaProp2}]
        Using Proposition~\ref{metaProp1}, it remains to show that
        $\liminf_{r \too r_\infty} \theta(\la(r), r) \ge \lim_{K \to \infty} a_K.$
        Note that
        $$\theta_{K s(r)}(\la(r), r) - \theta(\la(r), r) = \P(\infty \not\leftrightsquigarrow o \leftrightsquigarrow \partial Q_{Ks(r)}),$$
        so that
        \begin{align*} 
            \liminf_{r \too r_\infty} \theta(\la(r), r) &= \lim_{K \to \infty}a_{K} - \limsup_{K \to \infty}\limsup_{r \too r_\infty}\P(\infty \not\leftrightsquigarrow o \leftrightsquigarrow \partial Q_{Ks(r)}).
        \end{align*}
 Here the probability is understood to be formed under the Palm version of the Cox point process. By definition of the Palm version, we can bound this probability from above by the probability under the original Cox point process for the event
        $$E_{K, r} = \{Q_{1} \not \leftrightsquigarrow \infty\} \cap \{Q_1 \leftrightsquigarrow \partial Q_{(K-1)s(r)}\}.$$

    To construct a renormalized percolation process, a cube $Q_{s(r)}(s(r)z)$ is \emph{$r$-good} if
        \begin{enumerate}
            \item $R(Q_{5s(r)}(s(r)z)) < s(r)$, and
            \item $\vG(\vXl)\cap Q_{5s(r)}(s(r)z)$ contains a unique connected component intersecting both $\partial Q_{3s(r)}$ and $\partial Q_{5s(r)}$, and this component also intersects $Q_{s(r)}$.
        \end{enumerate}
    With a suitable choice of $q_d$, we conclude from~\cite[Theorem 0.0]{domProd} that the set of $r$-good sites is stochastically dominated from below by a Bernoulli percolation process with the following property: there exists an almost surely finite $K_0$ such that inside $Q_{K_0}$ there exists an interface of open sites separating $o$ from $\infty$ and that is contained in the infinite connected component. Note that, by condition (2), this infinite connected component also contains a unique infinite component in $\vG(\vXl)$. 
    Thus, under the event $E_{K, r}$ the infinite connected component of $r$-good sites does not contain an interface separating $o$ from $\infty$ in $Q_{K - 2}$.
      Hence,
        $$\limsup_{K \to \infty}\limsup_{r \too r_\infty}\P(E_{K,r}) \le \limsup_{K \to \infty}\P(K_0 \ge K - 2) = 0,$$
        as required.
\end{proof}

\subsection{Proof of Theorem~\ref{larThm}}
        We apply Propositions~\ref{metaProp1} and~\ref{metaProp2} with $a_K = \bar\theta_K(\r)$, $r_\infty = \infty$, $\la(r) = \r r^{-d}$ and $s(r) = nr$ for a suitably chosen $n \ge 1$.

\begin{proof}[Proof of Theorem~\ref{larThm}, convergence in bounded domains]
    We verify the condition of Proposition~\ref{metaProp1}. First, the connectivity properties of the Gilbert graph $\vG(\vXls)$ can be expressed equivalently in terms of the Poisson-Boolean model $\vXls \oplus B_{r/2}(o)$. Since the operation $A\mapsto A\oplus B_{r/2}$ is continuous in the space of compact sets~\cite[Theorem 12.3.5]{sWeil}, we conclude from a classical result from point process theory~\cite[Theorem 11.3.III]{daleyPPII2009} that the restriction of $\vXls \oplus B_{r/2}(o)$ to a bounded sampling window converges weakly to the corresponding restriction of a Poisson-Boolean model based on a Poisson point process with intensity $\r$. Since the indicator function that the origin is connected to the boundary of the box has discontinuities of measure 0 with respect to the Poisson-Boolean model, this yields convergence in bounded domains.
\end{proof}

\begin{proof}[Proof of Theorem~\ref{larThm}, tightness]
    By stabilization, the first condition in~\eqref{metaCond} is satisfied for sufficiently large $n \ge 1$ uniformly over all $r \ge 1$. Next, by~\cite[Theorem 2]{uniqFin}, if $\vXl$ is replaced by the limiting Poisson point process, then the second condition in~\eqref{metaCond} is satisfied for sufficiently large $n \ge 1$. This is true uniformly for all $r \ge 1$ by the scale invariance of the Poisson point processes. Finally, after fixing $n \ge 1$, weak convergence in finite domains implies that if $r$ is sufficiently large, then the second condition in~\eqref{metaCond} is satisfied also for $\vXl$.
\end{proof}

\subsection{Proof of Theorem~\ref{rlaAbsThm}}
We apply Propositions~\ref{metaProp1} and~\ref{metaProp2} with 
$$a_K = \E\big[\bar\theta(\r \ell_o^*)\one\{o\leftrightsquigarrow_\L\partial Q_K\}\big],$$
$r_\infty = 0$, $\la(r) = \r r^{-d}$ and $s(r) = n$  for a suitably chosen $n \ge 1$.  Recall that we assumed the intersection of any connected component of $L_{\ge}$ with $L_>$ to remain connected. 

\begin{proof}[Proof of Theorem~\ref{rlaAbsThm}, convergence in bounded domains]
    We show the condition for Proposition~\ref{metaProp1}. The proof of convergence in bounded domains proceeds along the lines of~\cite[Lemma 2.2]{adhoc}. We show that this convergence is true almost surely conditioned on the random field $\ell^*$. 
    For the convenience of the reader we reproduce the most important steps where we simply write $K$ instead of $K s(r)$ since $s(r)=n$. 

    %%%%NO GLOBAL CONNECTION
    Abusing notation, we consider $\theta_K(\la(r), r) = \P(o\leftrightsquigarrow_\L\partial Q_K | \L^*)$ as conditional probability. First, assume that $o\not\leftrightsquigarrow_\L\partial Q_K$. Then, in $Q_K$ by upper semicontinuity, the connected component $\mc{C}$ of $o$ in $L_\ge$ is a compact set of positive distance $\eta$ from $L_\ge \setminus \mc{C}$. In particular, writing
    $$\g= \{x \in Q_K:\, d_\infty(\mc{C}, x) = \eta/2 \},$$
    and $\eta' = \eta/4$, we see that $L_{\ge}$ does not intersect the discretized interface 
    $$\g^+ = \bigcup_{\substack{z \in \Z^d \\ Q_{\eta'}({\eta'} z) \cap \g \ne \es}} Q_{{3\eta'}}({\eta'} z).$$
    If the origin percolates beyond $\partial Q_K$ in $\vG(\vXls)$, then for sufficiently small $r$, the connected component crosses the interface.
    Hence, at least one of those cubes contains a connected set of $\vG(\vXls)$ of diameter at least $\eta$. By~\cite[Theorem 2]{uniqFin}, the probability of this event decays to 0 as $r \tod 0$ since in the cube there is a subcritical intensity.

    %%%%GLOBAL CONNECTION
    Hence, we may assume that $o\leftrightsquigarrow_\L\partial Q_K$. If we can show that for every $\e > 0$,
    $$\limsup_{r \tod 0} \theta_K(\la(r), r) \le \bar\theta(\r \ell_o^* + \e)$$
    and
    $$\liminf_{r \tod 0} \theta_K(\la(r), r) \ge \bar\theta(\r \ell_o^* - \e),$$
    then the result follows from assumption (1) and continuity of $\bar\theta$ away from $\bar\la_c$.     

    %%%%UPPER BOUND
    For the upper bound, we use assumption (2) and choose $\eta > 0$ sufficiently small such that $\sup_{x:\, |x|_\infty \le \eta} |\ell^*_x - \ell^*_o|\le\r^{-1} \e$. Then, by rescaling, 
    $$\limsup_{r \tod 0} \theta_K(\la(r), r) \le \limsup_{r \tod 0} \theta_\eta(\la(r), r)  \le \limsup_{r \tod 0}\bar\theta_{\eta /r}(\r \ell_o^* + \e) =\bar\theta(\r \ell_o^* + \e) .$$

    %%%%LOWER BOUND
    It remains to prove the lower bound. By assumption (3), there exists $\eta > 0$ and a path $\g$ connecting $o$ to $\partial Q_K$ in $L_>$ such that $\g_+ \subset L_>$, where 
    $$\g_{}^+ = \bigcup_{\substack{z \in \Z^d \\ Q_\eta(\eta z) \cap \g \ne \es}} Q_{3\eta}({\eta} z)$$
    denotes the discretization of $\g$. 
        Again, by~\cite[Theorem 2]{uniqFin} and stochastic domination, with high probability as $r\tod 0$, inside $\g^+$ there exists a unique connected component of $\vG(\vXls)$ with diameter at least $\eta/2$. Hence, 
        $$\liminf_{r \tod 0}\theta_K(\la(r), r) \ge \liminf_{r \tod 0}\theta_\eta(\la(r), r) \ge  \liminf_{r \tod 0}\bar\theta_{\eta /r}(\r \ell_o^* - \e)\ge \bar\theta(\r \ell_o^* - \e),$$
        as asserted.
\end{proof}

Recall that we additionally assumed that $\L$ is stabilizing and with a probability tending to 1 in $n$, the set 
    $L_> \cap Q_{5n} \setminus Q_{3n}$
    contains a compact interface separating $\partial Q_{3n}$ from $\infty$.

\begin{proof}[Proof of Theorem~\ref{rlaAbsThm}, tightness]
     By stabilization, the first condition in~\eqref{metaCond} is satisfied for all large $n \ge 1$. Moreover, by assumption, with high probability as $n \to \infty$, there exists a compact set 
     $$\g \subset \{x \in \R^d:\, \ell_x > \bar\la_c/\r\} \cap Q_{5n} \setminus Q_{3n}$$ 
     separating $\partial Q_{3n}$ from $\infty$. In particular, for $\eta$ sufficiently small also the discretized interface
    $$\g_{}^+ = \bigcup_{\substack{z \in \Z^d \\ Q_\eta(\eta z) \cap \g \ne \es}} Q_{3\eta}({\eta} z)$$
    satisfies $\inf_{x \in \g_+} \ell_x> \bar\la_c/\r$. Invoking~\cite[Theorem 2]{uniqFin} and stochastic domination once more, as $r \tod 0$, with high probability, inside $\g^+$ there exists a unique connected component of $\vG(\vXl)$ with diameter at least $\eta/2$. We write $E_r'(o)$ for this event, and more generally we write $E_r'(x)$ for the event that $E_r'(o)$ occurs for $\vXl-nx$. We obtain that 
    $$E_r'(o) \cap E_r'(2e_1) \subset E_r,$$
    see Figure~\ref{capFig}. Hence, also the second condition in~\eqref{metaCond} is satisfied with high probability.

    \begin{figure}[!htpb]
    \begin{tikzpicture}
        %CENTERS
            \coordinate [label=below:$o$] (p) at (2.5, 2.5);
            \coordinate [label=below:$2e_1$] (q) at (4.5, 2.5);

            \fill[black] (2.5, 2.5) circle (2pt);
            \fill[black] (4.5, 2.5) circle (2pt);

        %%BIG SQUARES
        \draw[dashed, ultra thick] (0,0) rectangle (5,5);
        \draw[dashed, ultra thick] (2,0) rectangle (7,5);

        %SUB-SQUARES
        \draw[dashed] (1,1) rectangle (4,4);
        \draw[dashed] (3,1) rectangle (6,4);

        %CYCLES
        \draw  plot [smooth cycle, tension = 1.2] coordinates { (0.5,0.5) (2,0.2) (3,0.8) (4.7,.3) (4.1, 2)   (4.6,4.6) (2,4.1) (0.4,4.6) (0.8, 2) };
        \draw  plot [smooth cycle, tension = 1.1] coordinates { (2.3,0.3) (4,0.8) (5,0.2) (6.5,.5) (6.1, 2)   (6.6,4.6) (5,4.1) (4,4.7) (2.2,4.0) (2.8, 2) };

        \end{tikzpicture}
        \caption{Illustration of the event $E_r'(o) \cap E_r'(2e_1)$.}
        \label{capFig}
    \end{figure}
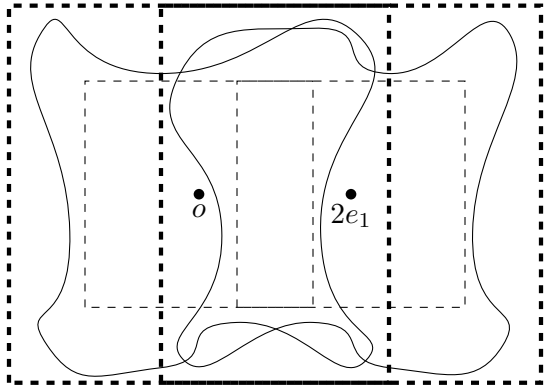
    
\end{proof}

\subsection{Proof of Theorem~\ref{rlaSingThm}}
 We apply Propositions~\ref{metaProp1} and~\ref{metaProp2} with $a_K = \theta_{\ms{Ber}, K}$, $r_\infty = 0$, $\la(r) \exp(-\la(r)r) = c$ and $s(r) = n$ for a suitably chosen value of $n \ge 1$.
\begin{proof}[Proof of Theorem~\ref{rlaSingThm}, convergence in bounded domains]
   We show the condition for Proposition~\ref{metaProp1}. For this, we prove convergence when conditioned on $S^* \cap Q_K$ and proceed in two steps. In the first step, we introduce an intermediate inhomogeneous Bernoulli bond percolation process on $S^* \cap Q_K$ where an edge $e$ is removed if there exist two successive points of $\vXls$ on $e$ of distance larger than $r$. We show that this intermediate percolation model converges weakly to the inhomogeneous Bernoulli percolation model on $S^* \cap Q_K$ where an edge $e$ is kept with probability $\exp(-c|e|)$. In particular, writing $\theta^{(r)}_K$ for the percolation probability of the intermediate model, 
    \begin{align}
        \label{weakSingEq}
    \lim_{r \tod 0} \theta^{(r)}_K = \theta_{\ms{Ber}, K}.
    \end{align}
    In the second step, we show that 
    \begin{align}
        \label{probSingEq}
    \lim_{r \tod 0} |\theta^{(r)}_K - \theta_K(\la(r), r)|= 0.
    \end{align}
    For~\eqref{weakSingEq} since both the intermediate and the target percolation model are Bernoulli bond percolation models, it suffices to prove convergence of the edge survival probabilities. Hence, let $e$ be an edge of $S^*$. 
    Since conditioned on $\L$, the points $\vXls$ form a Poisson point process, the distances between consecutive points on $e$ are iid $\ms{Exp}(\la(r))$-distributed random variables. In particular, the probability that the next Poisson point is of distance at most $r$ equals $1 - \exp(-\la(r) r)$.
    Moreover, by Poisson concentration the total number of points on $\vXls \cap e$ deviates at most by $\la(r)^{3/4}$ from the mean $\la(r) |e|$ with high probability. Thus, asymptotically the probability for $e$ to have no gaps of length at least $r$ is contained in the interval
    $$[(1 - \exp(-\la(r) r))^{\la(r) |e| + \la(r)^{3/4}}, (1-\exp(-\la(r) r))^{\la(r) |e| - \la(r)^{3/4}}].$$
    In the considered limiting regime, the upper and lower bounds converge to the common value
    $$\exp(-c|e|) = (\exp(-c))^{|e|}.$$

    To prove~\eqref{probSingEq}, we first note that with high probability on every edge there are Cox points at distance at most $r/2$ from the end points of the edge. In particular,
    $$\liminf_{r \tod 0} \theta_K(\la(r), r) - \theta^{(r)}_K \ge 0.$$
    For the upper bound, we fix $M \ge 1$ such that when writing $V_K$ for the set of vertices in $S^* \cap Q_K$, then $S^* \setminus (\cup_{v \in V_K} B_{Mr}(v))$ decomposes into connected components of mutual distance at least $r$. Then, with high probability for every $v \in V_K$ and $u \in B_{Mr/2}(v) \cap S^*$ there exists a Cox point at distance at most $r/2$ from $u$. In particular, if $\vG(\vXls)$ percolates beyond $\partial Q_K$, then so does the intermediate model. In other words,
    $$\limsup_{r \tod 0} \theta_K(\la(r), r) - \theta^{(r)}_K \le 0,$$
    as required.
\end{proof}

\begin{proof}[Proof of Theorem~\ref{rlaSingThm}, tightness]
    As usual, by stabilization, the first condition in~\eqref{metaCond} is satisfied for sufficiently large $n \ge 1$. Moreover, by definition of the intermediate percolation model introduced above, if $c$ is small, then asymptotically, the probability that in this model, no edges are removed in the cube $Q_{5n}$ is larger than $q_d$ provided that $r$ is sufficiently small. Moreover, by asymptotic essential connectedness, with high probability, $S\cap Q_{5n}$ contains a unique connected component intersecting both $\partial Q_{3n}$ and $\partial Q_{5n}$, and this component also intersects $Q_{n}$. In particular, also the second condition in~\eqref{metaCond} is satisfied for all sufficiently small $r > 0$. 
\end{proof}

 \subsection{Acknowledgments}
 This research was supported by the Leibniz program \emph{Probabilistic Methods for Mobile Ad-Hoc Networks}, by LMU Munich's Institutional Strategy LMUexcellent within the framework of the German Excellence Initiative and by Orange~S.A.~grant CRE G09292. The authors thank Nila Novita Gafur for her contribution in the simulation section.

\bibliography{../../wias}
\bibliographystyle{abbrv}

\end{document}